\documentclass[11pt]{article}
\usepackage{amsfonts}
\usepackage{amssymb}
\usepackage{amsthm}
\usepackage{amsmath}
\usepackage{mathrsfs}
\usepackage{bm}
\usepackage{amscd} 
\usepackage{fullpage}
\usepackage{hyperref}
\usepackage{xypic}
\numberwithin{equation}{section} \DeclareMathSizes{2}{10}{12}{13}
\usepackage[all]{xy}

\input xy
\xyoption{all}
\newtheorem{thm}{Proposition}[section]

\newtheorem{rem}[thm]{Remark}

\newtheorem{cor}[thm]{Corollary}
\newtheorem{lem}[thm]{Lemma}
\newtheorem{defn}[thm]{Definition}

\title{Noetherian schemes over abelian symmetric monoidal categories
}
\author{Abhishek Banerjee
}
\date{ }

\begin{document}

\maketitle

\medskip

\medskip
\centerline{\emph{Max Planck Institut f\"{u}r Mathematik, Vivatsgasse 7, Bonn, Germany.}}

\centerline{\emph { Email: abhishekbanerjee1313@gmail.com}}

\medskip
\begin{abstract} In this paper, we develop  basic results of algebraic geometry over abelian
symmetric monoidal categories. Let $A$ be a commutative monoid object in an abelian symmetric monoidal category $(\mathbf C,\otimes,1)$ satisfying 
certain conditions and let $\mathcal E(A)=Hom_{A-Mod}(A,A)$. If the subobjects of $A$ satisfy a certain compactness property, we say that $A$
is  Noetherian. We study the localisation 
of $A$ with respect to any $s\in \mathcal E(A)$ and define the quotient $A/\mathscr I$ of $A$ with respect
to any ideal $\mathscr I\subseteq \mathcal E(A)$. We use this to develop appropriate  analogues of the basic  notions from usual
 algebraic geometry (such as Noetherian schemes, irreducible, integral and reduced schemes, function field, the local ring
 at the generic point of a closed subscheme, etc) for schemes over  $(\mathbf C,\otimes,1)$ . Our  notion of a scheme over a symmetric monoidal category $(\mathbf C,\otimes,1)$ is that
 of To\"{e}n and Vaqui\'{e}. \end{abstract}
 
 \medskip
 
 \medskip
\emph{ MSC (2010) Subject Classification: 14A15, 19D23}


\medskip

\medskip

\section{Introduction}

\medskip

\medskip 
The relative algebraic geometry over a symmetric monoidal
category $(\mathbf C,\otimes,1)$ has been  studied at several places 
in the literature (see, for instance, Deligne \cite{Delg}, 
Hakim \cite{Hak}, To\"{e}n
and Vaqui\'{e} \cite{Toen2}). When  $\mathbf C=R-Mod$, the category
of modules over an ordinary commutative ring $R$, this reduces to
the usual algebraic geometry of schemes over $Spec(R)$. In this paper, we will develop
basic results of commutative algebra and algebraic geometry over abelian symmetric
monoidal categories satisfying certain conditions.  For instance, our methods
enable us to do algebraic geometry in the category of presheaves of abelian
groups over a topological space. This paper continues our research program
to study monoid objects and schemes over symmetric monoidal categories 
(see \cite{AB14}, \cite{AB12}, \cite{2AB2}, \cite{0AB1}, \cite{1AB1}, \cite{3AB3}). 

\medskip More precisely, let $(\mathbf C,\otimes,1)$ be an abelian  symmetric monoidal category satisfying certain conditions
described in Section 2. We will use the notion of schemes over $(\mathbf C,\otimes,1)$
introduced by To\"{e}n and Vaqui\'{e} \cite{Toen2}.  It is natural to ask if we can develop in detail the results of intersection theory
for schemes over $(\mathbf C,\otimes,1)$. A starting point for this is 
 to define appropriate  analogues of the basic  notions from usual
 algebraic geometry (such as Noetherian schemes, irreducible, integral and reduced schemes, function field, the local ring
 at the generic point of a closed subscheme, etc) for schemes over a symmetric monoidal category $(\mathbf C,\otimes,1)$. This is the aim of the present paper. For our purposes, we will also need to develop some commutative algebra
over $(\mathbf C,\otimes,1)$ and this utilises the notion of localisation of  commutative monoid objects introduced in \cite{AB14}. For the sake of convenience, the main properties of this localisation developed in
\cite{AB14} are recalled briefly in Section 2.

\medskip
Let $Comm(\mathbf C)$ denote the category of commutative monoid objects in $(\mathbf C,\otimes,1)$. For any commutative
monoid object $A$, we let $A-Mod$ denote the category of $A$-modules. Following \cite{Toen2}, we set $Aff_{\mathbf C}:=Comm(\mathbf C)^{op}$ to be the
category of affine schemes over $\mathbf C$. The affine scheme
corresponding to a commutative monoid object $A$ will be denoted by
$Spec(A)$. Given an element $s\in \mathcal E(A):=Hom_{A-Mod}(A,A)$,
we consider the localisation $A_s$ of $A$ introduced in 
\cite{AB14}. Then, in Section 2, we show that any morphism
$Spec(A_s)\longrightarrow Spec(A)$ is a Zariski open immersion. 
Further, we prove that a collection $\{Spec(A_{t_i})\longrightarrow
Spec(A)\}_{t_i\in \mathcal E(A),i\in I}$ of Zariski open immersions forms a cover
of $Spec(A)$ if and only if $\{t_i\}_{i\in I}$ generate the unit ideal
in the ring $\mathcal E(A)$. 

\medskip
We start working with Noetherian schemes (see Definition \ref{revD3.6}) in Section 3.  A commutative
monoid object $A$ is said to be Noetherian if its subobjects (in $A-Mod$) satisfy a certain compactness property 
(see Definition \ref{revD3.1}). As in 
usual algebraic geometry, we prove that being Noetherian is a
local property of schemes. Given a Noetherian monoid $A$
and any ideal $\mathscr I\subseteq \mathcal E(A)$, we introduce
a ``quotient monoid'' $A/\mathscr I$ which is used for the 
construction of closed subschemes in Section 5. If $A$ is Noetherian, so is the
 quotient monoid $A/\mathscr I$ and
 the canonical morphism $p:A\longrightarrow A/\mathscr I$
is an epimorphism in the category $Comm(\mathbf C)$. 
Moreover, we show that if $A$ is a Noetherian 
monoid object, $\mathcal E(A)$ is an ordinary Noetherian commutative ring and  
$\mathcal E(A/\mathscr I)=\mathcal E(A)/\mathscr I$. 

\medskip
We consider integral schemes in Section 4. We show that an integral scheme $X$ over $(\mathbf C,\otimes,1)$ is  reduced and irreducible. For our purposes, we will need to consider 
a second, related notion of integrality that we shall refer to as ``weak integrality''. 
We show that a reduced and irreducible scheme is weakly integral. We then 
associate to any integral scheme $X$, a field $k(X)$ that plays
the role of function field  in the context of schemes over
$(\mathbf C,\otimes,1)$. Thereafter, given a dominant morphism
$f:Y\longrightarrow X$ of integral schemes over $\mathbf C$, we construct 
an induced morphism $k(f):k(X)\longrightarrow k(Y)$ 
of function fields. 

\medskip
Finally, in Section 5, we construct closed subschemes of a  Noetherian
and semi-separated scheme. More generally, we show that there is a
one-one correspondence between quasi-coherent sheaves of 
algebras on a semi-separated scheme $X$ and affine morphisms
$Y\longrightarrow X$ (see Definition \ref{def4.2} and Proposition
\ref{affinemorp}). In particular, when we have a quasi-coherent
sheaf of quotient monoids on a Noetherian and semi-separated scheme $X$, the corresponding affine morphism
$Y\longrightarrow X$ gives us a closed subscheme $Y$ of $X$. Further, we
show that the closed subscheme $Y$ of $X$ is also Noetherian. Finally, to any integral closed subscheme $Y$ 
of a Noetherian, integral and semi-separated scheme $X$, we associate
a local ring $\mathcal O_Y$. In usual algebraic geometry, $\mathcal 
O_Y$ is the local ring at the generic point of the 
integral closed
subscheme $Y$. 

\medskip
Section 6 is devoted to examples. If $X$ is a topological space and $\mathcal A$
is a presheaf of commutative rings on $X$, we show that our theory can be used
to do algebraic geometry in the category of presheaves of $\mathcal A$-modules
on $X$. We then use this fact to give several natural examples of our theory.

\medskip

\medskip

{\bf Acknowledgements:} I gratefully acknowledge support from IH\'{E}S, Bures-sur-Yvette, where part of this paper was written.

\medskip

\medskip

\section{Coverings of affine schemes}

\medskip

\medskip 
Let $(\mathbf C,\otimes,1)$ be an abelian  symmetric monoidal category. We assume
that $\mathbf C$ contains small limits and small colimits and for any object $X\in 
\mathbf C$, the functor $\_\_\otimes X$ preserves colimits. We let $Comm(\mathbf C)$ denote the category of unital commutative monoids in $\mathbf C$. For an object $A$ in $Comm(\mathbf C)$, we will always denote by $m_A:A\otimes A\longrightarrow A$ the ``multiplication map'' and by $e_A:1\longrightarrow A$ the ``unit map'' on $A$. Further, for any object $A$ in $Comm(\mathbf C)$, we will denote by $A-Mod$ the category of modules over $A$.  For generalities on monoids and modules over them in symmetric monoidal categories, we refer the reader to \cite{ML2}. All monoid objects considered in this paper shall be assumed to be unital and commutative.  Further, for any monoid object $A$, we will assume that filtered
colimits commute with finite limits in $A-Mod$. This latter assumption 
is key to the results on localisations in \cite{AB14}, which we
shall use throughout this paper. 

\medskip
Since $\mathbf C$ is an abelian category, finite products and 
finite coproducts in $\mathbf C$ coincide. For any object
$X$ in $\mathbf C$ and any integer $r> 0$ we let $X^r$ denote the
finite product (or coproduct) of $r$-copies of $X$.

\medskip
We will also assume that the category $\mathbf C$ satisfies the following two technical conditions:

\medskip
(C1) The unit object $1$ is compact, i.e., the functor $Hom(1,\_\_)$ on $\mathbf C$ preserves filtered colimits. Since $Hom(1,M)\cong Hom_{A-Mod}(A,M)$ for any
monoid $A$ and any object $M$ in $A-Mod$, it follows 
that $A$ is a compact object of $A-Mod$.  Further, we assume that given a finite system (not necessarily filtered) of objects of the form $\{A^{r_i}\}_{i\in I}$,
$r_i\geq 0$, we have 
\begin{equation}\label{rev2.1}
colim_{i\in I}\textrm{ }Hom_{A-Mod}(A,A^{r_i})\cong Hom_{A-Mod}(A,colim_{i\in I}
A^{r_i})
\end{equation} Again, we note that \eqref{rev2.1} is equivalent to assuming 
that $\underset{i\in I}{colim}\textrm{ }Hom(1,A^{r_i})\cong Hom(1,\underset{i\in I}{colim}
A^{r_i})$. 

\medskip
(C2) If  $A\in Comm(\mathbf C)$ is a commutative monoid, an object $M\in A-Mod$ will be said to 
be finitely presented if the functor $Hom_{A-Mod}(M,\_\_)$ on $A-Mod$ preserves
directed colimits. We will assume that for any commutative monoid $A$ in $\mathbf C$, every object in $A-Mod$ may be expressed as a directed colimit of finitely presented objects in $A-Mod$.

\medskip
\begin{rem}\emph{ The condition (C2) above may be seen as an analogue of the fact that in the category of modules over an ordinary commutative ring, any module may be expressed as a directed colimit of finitely generated submodules. In fact, the condition (C2) implies that the category $A-Mod$ is ``locally
finitely presented''. The theory of locally finitely presentable
categories is fairly well developed in the literature and one may 
see, for instance, \cite{Gar}, \cite{0Prest}, \cite{Prest}, \cite{1Prest}, \cite{Sten}. In \cite{0AB1}, we have also studied module
categories for monoid objects in symmetric monoidal categories
under the somewhat similar assumption of ``locally finitely generated''.
For generalities on locally finitely generated and locally finitely
presentable categories, we refer the reader to \cite{AdRos}.}
\end{rem}

\medskip
To any unital, commutative monoid object $A$ in $(\mathbf C,\otimes,1)$, we  can associate the object $\mathcal E(A):=Hom_{A-Mod}(A,A)$ of $A$-module morphisms from $A$ to $A$. It is well known, see, for instance
\cite{MayK}, that $\mathcal E(A)$ is a commutative ring. Given a morphism $g:A\longrightarrow B$ in $Comm(\mathbf C)$, it follows from base change that we have an induced morphism $\mathcal E(g):\mathcal E(A)\longrightarrow \mathcal E(B)$ of commutative rings.

\medskip
\noindent Let $A$ be a  monoid object in $(\mathbf C,\otimes,1)$ and let us choose any $t\in \mathcal E(A)$. Then, in \cite[$\S$3]{AB14}, we have defined a commutative monoid object $A_t$  as follows:
\begin{equation}
A_t:=colim(A\overset{t}{\longrightarrow}A \overset{t}{\longrightarrow} A
\overset{t}{\longrightarrow}\dots)
\end{equation} which we call the localisation of $A$ with respect to $t$. More generally, if $S\subseteq \mathcal E(A)$ is a ``multiplicatively
closed subset'', i.e., the identity map $1_A\in S$ and for any $s$, $t\in S$, the composition $s\circ t=t\circ s\in S$, we have defined the localisation
\begin{equation}\label{vprim}
A_S:=\underset{s\in S}{colim}\textrm{ }A_s
\end{equation}in \cite{AB14}. We note here that since $S$ is closed under composition, the colimit in \eqref{vprim} is filtered. The object $A_S$ is equipped with a canonical morphism $I_S:A\longrightarrow A_S$ of monoids. For any $A$-module $M$, the localisation of $M$ with respect to $S$ is defined to be $M_S:=M\otimes_AA_S$. Further, we have shown in \cite{AB14} that the localisation $A_S$ satisfies the following properties:

\medskip
\noindent (a) $A_S$ is a flat $A$-module, i.e., the functor $\_\_\otimes_AA_S$ on $A-Mod$ preserves finite limits and finite colimits.

\medskip
\noindent (b) Consider the morphism $\mathcal E(I_S):\mathcal E(A)\longrightarrow \mathcal E(A_S)$ induced by the morphism $I_S:A\longrightarrow A_S$. Then, for any $s\in S$, the morphism $\mathcal E(I_S)(s)\in \mathcal E(A_S)$ 
 is an isomorphism. Further, given any morphism $g:A\longrightarrow B$ in $Comm(\mathbf C)$ such that $\mathcal E(g)(s)\in \mathcal E(B)$ is an isomorphism for each $s\in S$, there exists a unique morphism $h:A_S\longrightarrow B$
such that $g=h\circ I_S$.

\medskip We note that property (b) above implies that 
the canonical morphism $I_S:A\longrightarrow A_S$ 
is an epimorphism in the category $Comm(\mathbf C)$, i.e., given any
morphisms $f_1,f_2:A_S\longrightarrow B$ in $Comm(\mathbf C)$ such
that $f_1\circ I_S=f_2\circ I_S$, we must have $f_1=f_2$.

\medskip
\noindent Let $Aff_{\mathbf C}=Comm(\mathbf C)^{op}$  denote the category of affine schemes over $\mathbf C$. If $A$ is an object of $Comm(\mathbf C)$, we will often use $Spec(A)$ to denote the corresponding object in
$Aff_{\mathbf C}$. Then,  To\"{e}n and Vaqui\'{e} (see \cite[D\'{e}finition 2.10]{Toen2}) have introduced the notion of Zariski coverings in the category $Aff_{\mathbf C}$, determining a  Grothendieck site that is also subcanonical, i.e. the representable presheaves on $Aff_{\mathbf C}$ are also sheaves. Accordingly, let $Sh(Aff_{\mathbf C})$ denote the category of sheaves of sets on $Aff_{\mathbf C}$. Further, in \cite[D\'{e}finition 2.12]{Toen2}, To\"{e}n and Vaqui\'{e} have introduced a suitable notion of Zariski open immersions in the category $Sh(Aff_{\mathbf C})$ that is stable under composition and base change. Then, a  scheme $X$ over $\mathbf C$ is defined to be  an object of $Sh(Aff_{\mathbf C})$ admitting a Zariski covering by affine schemes (see \cite[D\'{e}finition 2.15]{Toen2}). 
By abuse of notation, we will often denote  the sheaf on $Aff_{\mathbf C}$ represented
by a scheme $X$ (resp. an affine scheme $Spec(A)$) also by $X$ (resp.  $Spec(A)$). 

\medskip  Given a monoid $A$, in this section, our aim is to
study Zariski coverings of $Spec(A)$ by means of schemes
of the form $\{Spec(A_t)\}_{t\in \mathcal E(A)}$. We recall here the notion of a Zariski open immersion of affine schemes over $(\mathbf C,\otimes,1)$ as defined in \cite[D\'{e}finition 2.9]{Toen2}.

\medskip
\begin{defn}\label{zarop} Let $f:A\longrightarrow B$ be a morphism 
in $Comm(\mathbf C)$. 

\medskip
(1) The morphism $f$ is flat if the functor $\_\_\otimes_AB:A-Mod\longrightarrow B-Mod$ is exact, i.e., preserves finite limits. 

\medskip
(2) The morphism $f$ is an epimorphism if, for all $A'$ in $Comm(\mathbf C)$, the induced morphism $f^*:Hom_{Comm(\mathbf C)}(B,A')\longrightarrow Hom_{Comm(\mathbf C)}(A,A')$ is injective. 

\medskip
(3) The morphism $f$ is of finite presentation if for any filtered
system  of objects $A_i'\in A/Comm(\mathbf C)$, $i\in I$ the natural isomorphism
\begin{equation}
colim_{i\in I}Hom_{A/Comm(\mathbf C)}(B,A_i')\longrightarrow Hom_{A/Comm(\mathbf C)}(B,colim_{i\in I}A_i') 
\end{equation} is an isomorphism. 

\medskip
(4) The morphism $Spec(B)\longrightarrow Spec(A)$ induced by
$f$  is a Zariski open immersion if $f$ is a flat epimorphism of finite presentation. 

\medskip
(5) Given morphisms $f_i:A\longrightarrow B_i$, $i\in I$, the collection of functors $\{\_\_\otimes_AB_i:
A-Mod\longrightarrow B_i-Mod\}_{i\in I}$ is said to be conservative if, for any $M$ in $A-Mod$, $M=0$ if and only
if $M\otimes_AB_i=0$ for each $i\in I$. 

\medskip
(6) Consider morphisms $f_i:A\longrightarrow B_i$, $i\in I$ such that there exists a finite subcollection $I'\subseteq I$ such that the  family of functors $\{\_\_\otimes_AB_i:
A-Mod\longrightarrow B_i-Mod\}_{i\in I'}$ is conservative. Then, if each $f_i:A\longrightarrow B_i$, $i\in I$ induces a Zariski open
immersion of affine schemes, $\{Spec(B_i)\longrightarrow Spec(A)\}_{i\in I}$ is said to be a Zariski open cover
of $Spec(A)$. 

\end{defn} 

\medskip
We also recall here the definition of a scheme over $(\mathbf C,\otimes,1)$ due to To\"{e}n and Vaqui\'{e}  (see 
\cite[D\'{e}finition 2.15]{Toen2}).

\begin{defn}Let $X$ be an object of $Sh(Aff_{\mathbf C})$. Then, $X$ is a scheme over $(\mathbf C,\otimes,1)$ if there exists a family $\{X_i\}_{i\in I}$ of affine schemes over $(\mathbf C,\otimes,1)$ and a morphism 
\begin{equation}
p:\coprod_{i\in I}X_i\longrightarrow X
\end{equation} satisfying the following conditions: 

\medskip
(a) The morphism $p$ is an epimorphism in $Sh(Aff_{\mathbf C})$. 

\medskip
(b) For each $i\in I$, the morphism $X_i\longrightarrow X$ is a Zariski open immersion in $Sh(Aff_{\mathbf C})$. 
\end{defn} 

\medskip

\medskip
\begin{lem}\label{compact} Let $A_i$, $i\in I$ be a filtered system of objects in $Comm(\mathbf C)$ and let $A=colim_{i\in I}A_i$. Then, we have $\mathcal E(A)=colim_{i\in I}\mathcal E(A_i)$. 
\end{lem} 

\begin{proof}   By assumption (C1), $1$ is a compact object of $\mathbf C$ and it follows therefore that:
\begin{equation*} \mathcal E(A)=Hom_{A-Mod}(A,A)\cong Hom(1,A)\cong colim_{i\in I}Hom(1,A_i)\cong Hom_{A_i-Mod}(A_i,A_i)=\mathcal E(A_i)
\end{equation*} This proves the result. 
\end{proof}

\medskip
\begin{thm}\label{openimmer} Let $A$ be an object of $
Comm(\mathbf C)$. We choose some $t\in \mathcal E(A)$ and let $I_t:A\longrightarrow A_t$ denote the localisation of $A$ with respect to $t$. Then, the induced morphism $Spec(A_t)\longrightarrow Spec(A)$ is a Zariski open immersion of schemes. 
\end{thm}

\begin{proof} We have to verify conditions (1), (2) and (3) in
Definition \ref{zarop} for the morphism $f$. We have already mentioned that the functor
$\_\_\otimes_AA_t$ preserves finite limits.  Similarly,
we have also mentioned before that any morphism $I_t:A
\longrightarrow A_t$ induced by a localisation is an epimorphism in $Comm(\mathbf C)$. It follows that $f$ satisfies conditions (1) and (2) of Definition \ref{zarop}.

\medskip
Finally, we consider a filtered system $A_i'\in A/Comm(\mathbf C)$, $i\in I$ and we set $A':=colim_{i\in I}A_i'$. By definition, each object $A_i'\in A/Comm(\mathbf C)$ is equipped with canonical morphisms
\begin{equation}h_i:A\longrightarrow A_i'\qquad g_i:A_i'\longrightarrow A'
\end{equation} such
that $g_i\circ h_i=g_j\circ h_j$ $\forall$ $i,j\in I$. We consider a morphism $g:A_t\longrightarrow A'$ in $A/Comm(\mathbf C)$.  Then,
\begin{equation}
g\circ I_t=g_i\circ h_i:A\longrightarrow A' \qquad \forall \textrm{ }i\in I 
\end{equation} Hence, for all $i\in I$, 
\begin{equation}\mathcal E(g_i\circ h_i)=\mathcal E(g\circ I_t)(t)
=\mathcal E(g)(\mathcal E(I_t)(t))\end{equation} is a unit in $\mathcal E(A')$. From Lemma \ref{compact}, we have $\mathcal E(A')=colim_{i\in I}\mathcal E(A_i')$ and hence there exists $i_0\in I$ such that $\mathcal E(h_{i_0})(t)$ is a unit in $\mathcal E(A_{i_0}')$. It now follows that there exists a morphism
$h':A_t\longrightarrow A_{i_0}'$ such that $h_{i_0}=h'\circ I_t$. It follows that $f$ satisfies condition (3) of Definition \ref{zarop}.

\end{proof}

\medskip
\begin{lem}\label{local} (a) Let $A$ be a commutative monoid and let $M$ be an $A$-module. Let $\{t_i\}_{i\in I}$ be a finite collection of elements
$t_i\in \mathcal E(A)$ such that $\sum_{i\in I}t_i=1$ and let 
$M_i$ denote the respective localisations $M_i:=M_{t_i}$, $\forall$ 
$i\in I$. Then, if each $M_i=0$, we must have $M=0$. 

\medskip
(b)  Let $A$ be a commutative monoid and let $M$ be an $A$-module. Let $\{t_i\}_{i\in I}$ be a finite collection of elements
$t_i\in \mathcal E(A)$ such that there exists a collection $\{s_i\}_{i\in I}$, $s_i\in \mathcal E(A)$ such that $\sum_{i\in I}s_it_i=1$. Let 
$M_i$ denote the respective localisations $M_i:=M_{t_i}$, $\forall$ 
$i\in I$. Then, if each $M_i=0$, we must have $M=0$. 
\end{lem}

\begin{proof} (a) For any $t\in \mathcal E(A)$, we denote the induced morphism $M\cong M\otimes_AA\overset{1\otimes t}{\longrightarrow}
M\otimes_AA\cong M$ by $t_M$. By definition, we know that
\begin{equation}\label{1.6}
M_t:=M\otimes_AA_t=colim(M\overset{t_M}{\longrightarrow}M\overset{t_M}{\longrightarrow}M\overset{t_M}{\longrightarrow}\dots) 
\end{equation} (since $M\otimes_A\_\_$ commutes with  colimits). Let $N$ be a finitely
presented object in $A-Mod$. We note that $Hom_{A-Mod}(N,M)$ can be made into an $\mathcal E(A)$-module as follows: given $f\in Hom_{A-Mod}(N,M)$, $t\in \mathcal E(A)$, we set $f\cdot t:=t_M\circ f\in Hom_{A-Mod}(N,M)$. 

\medskip In particular, it is clear that the induced morphism $Hom_{A-Mod}(N,t_M):Hom_{A-Mod}(N,M)\longrightarrow Hom_{A-Mod}(N,M)$ is identical to multiplication by $t\in \mathcal E(A)$ on the
$\mathcal E(A)$-module $Hom_{A-Mod}(N,M)$. Since $N$ is a finitely presented object of $A-Mod$, it follows that
\begin{equation}\label{1.7}
\begin{array}{l}
Hom_{A-Mod}(N,M_t)\\ 
 \cong colim_{i\in I}\left(\begin{CD}Hom_{A-Mod}(N,M)@>Hom_{A-Mod}(N,t_M)>> Hom_{A-Mod}(N,M) @>Hom_{A-Mod}(N,t_M)>> \dots\end{CD}\right) \\
 \cong colim_{i\in I}\left(\begin{CD}Hom_{A-Mod}(N,M)@>\cdot t>> Hom_{A-
 Mod}(N,M) @>\cdot t>> \dots\end{CD}\right) \\
 \cong Hom_{A-Mod}(N,M)_{t} 
\end{array}
\end{equation} where $Hom_{A-Mod}(N,M)_{t}$ denotes the localisation
of the $\mathcal E(A)$-module $Hom_{A-Mod}(N,M)$ with respect to $t\in 
\mathcal E(A)$. 

\medskip
Hence, if $M_i=M_{t_i}=0$ for all $i\in I$, it follows from \eqref{1.7} that for any finitely presented object $N$ in $A-Mod$, $Hom_{A-Mod}(N,M)_{t_i}=0$. Since $\sum_{i\in I}t_i=1$ in $\mathcal E(A)$, it follows that
$Hom_{A-Mod}(N,M)=0$ for any finitely presented object $N$ in $A-Mod$. Finally, since any object in $A-Mod$ can be expressed as a colimit of finitely presented objects (using condition (C2)), it follows that $Hom_{A-Mod}(N,M)=0$ for any object $N$ in $A-Mod$. Hence, $M=0$. 

\medskip
(b) By interchanging colimits, it follows from \eqref{1.6} that
for any $i\in I$, $M_{s_it_i}=(M_{t_i})_{s_i}$. Then, since each
$M_{t_i}=0$, we have $M_{s_it_i}=0$ for each $i\in I$. It now follows from part (a) that $M=0$. 

\end{proof}

\medskip
Henceforth, given a monoid $A$, we will say that a finite collection
$\{t_i\}_{i\in I}$ of elements $t_i\in \mathcal E(A)$ is a partition
of unity on $A$ if there exist elements $\{s_i\}_{i\in I}$, $s_i\in \mathcal E(A)$ such that $\sum_{i\in I}s_it_i=1$.

\medskip
\begin{thm}\label{conservative} Let $A$ be a commutative monoid and let $u:M\longrightarrow N$ be a morphism of $A$-modules. Let $\{t_i\}_{i\in I}$ be a partition of unity on $A$. For any $i\in I$, let us denote by $u_i:M_i:=M_{t_i}\longrightarrow N_i:=N_{t_i}$ the induced morphisms on the localisations
of $M$ and $N$ with respect to $t_i$. Then, $u:M\longrightarrow N$ is an isomorphism if and only if each $u_i:M_i\longrightarrow N_i$, $i\in I$ is an isomorphism. 
\end{thm}

\begin{proof} The ``only if'' part of the result is clear. Conversely, suppose that each $u_i:M_i\longrightarrow N_i$ is an isomorphism. We consider the objects $Ker(u)$ and $Coker(u)$ in $A-Mod$ defined as follows:
\begin{equation}
Ker(u):=lim(M\overset{u}{\longrightarrow} N \longleftarrow 0) \qquad 
Coker(u):=colim(N\overset{u}{\longleftarrow} M \longrightarrow 0)
\end{equation} Since each $A_i:=A_{t_i}$ is a flat $A$-module, i.e., the functor $\_\_\otimes_AA_i$ preserves finite limits and finite colimits, it follows that for any $i\in I$, we have:
\begin{equation}\label{1.9}
\begin{array}{l} 
Ker(u_i)= lim(M_i\overset{u_i}{\longrightarrow} N_i \longleftarrow 0) 
 = lim (M\otimes_AA_i \overset{u_i}{\longrightarrow} N\otimes_AA_i \longleftarrow 0)
 = Ker(u)_{t_i} \\
\end{array} 
\end{equation} Similarly, for any $i\in I$, we have 
\begin{equation}\label{1.9XYR}
\begin{array}{l} 
Coker(u_i)= lim(M_i\overset{u_i}{\longrightarrow} N_i \longleftarrow 0) 
 = lim (M\otimes_AA_i \overset{u_i}{\longrightarrow} N\otimes_AA_i \longleftarrow 0)
 = Coker(u)_{t_i} \\
\end{array} 
\end{equation}
Since each $u_i$ is an isomorphism, we have $Ker(u_i)=Coker(u_i)=0$ for each $i\in I$. It follows from \eqref{1.9} that $Ker(u)_{t_i}=
Coker(u_{t_i})=0$ for each $i\in I$. Combining with Lemma \ref{local}, it follows that $Ker(u)=Coker(u)=0$. Since $A-Mod$ is an abelian
category, $u:M\longrightarrow N$ is an isomorphism. 
\end{proof}

\medskip
\begin{cor}\label{cor1} Let $A$ be a commutative monoid object in $\mathbf C$ and let $S\subseteq \mathcal E(A)$ be a multiplicatively closed set. Then, $\mathcal E(A_S)=\mathcal E(A)_S$. 
\end{cor}

\begin{proof} From condition (C1), we know that $A$ is a compact object of $A-Mod$. Then, it follows from \eqref{1.7}  that $Hom_{A-Mod}(A,A_s)\cong Hom_{A-Mod}(A,A)_s\cong \mathcal E(A)_s$ for any $s\in \mathcal E(A)$. Hence, we have:
\begin{equation}\label{aftert}
\begin{array}{r}
\mathcal E(A_S)\cong Hom_{A_S-Mod}(A_S,A_S) 
\cong Hom_{A-Mod}(A,A_S) 
\cong Hom_{A-Mod}(A,\underset{s\in S}{colim}A_s)\\
\cong \underset{s\in S}{colim}\textrm{ }Hom_{A-Mod}(A,A_s)
\cong \underset{s\in S}{colim}\textrm{ }\mathcal E(A)_s
\cong \mathcal E(A)_S\\
\end{array}
\end{equation} where  the isomorphism $Hom_{A-Mod}(A,\underset{s\in S}{colim}A_s)
\cong \underset{s\in S}{colim}\textrm{ }Hom_{A-Mod}(A,A_s)$
in \eqref{aftert} follows from the fact that $A$ is compact 
in $A-Mod$.

\end{proof}

\medskip
\begin{thm}\label{prop15} Let $A$ be a commutative monoid object in $\mathbf C$ and let  $\{t_i\}_{i\in I}$ be a partition of unity on $A$. Then, the schemes $\{Spec(A_{t_i})\}_{i\in I}$ form a Zariski open cover of $Spec(A)$. 

\end{thm} 

\begin{proof}  From Proposition \ref{openimmer}, we know that each morphism $Spec(A_{t_i})\longrightarrow Spec(A)$ is a Zariski open
immersion. Further, from Proposition \ref{conservative}, we know that
the collection of functors $\{\_\_\otimes_AA_{t_i}:A-Mod\longrightarrow A_{t_i}-Mod\}_{i\in I}$ is conservative. It follows that the collection $\{Spec(A_{t_i})\longrightarrow Spec(A)\}_{i\in I}$ is a Zariski open cover of $Spec(A)$. 

\end{proof} 

\medskip
 We conclude this section by proving the converse of Proposition \ref{prop15}.
Given a monoid $A$ and some $t\in \mathcal E(A)$, we define:
\begin{equation}\label{110}
A/tA:=colim(A\overset{t}{\longleftarrow}  A\longrightarrow 0) 
\end{equation} It is easy to check that $A/tA$ is also a commutative monoid. Further, using assumption (C1), we have
\begin{equation}\label{111}
Hom_{A-Mod}(A,A/tA)\cong colim(Hom_{A-Mod}(A,A)\overset{\cdot t}{\longleftarrow}  Hom_{A-Mod}(A,A)\rightarrow 0) \cong \mathcal E(A)/(t)
\end{equation} where $(t)$ in \eqref{111} denotes the principal ideal in $\mathcal E(A)$ generated by $t\in \mathcal E(A)$. It follows that:
\begin{equation}\label{112}
\mathcal E(A/tA)=Hom_{A/tA-Mod}(A/tA,A/tA)\cong Hom_{A-Mod}(A,A/tA)\cong \mathcal E(A)/(t)
\end{equation} More generally, if $\{t_1,...,t_n\}$ is a set of 
elements in $\mathcal E(A)$,  we know from \eqref{110}
and \eqref{112} that the monoid
$A/t_1A$ has $\mathcal E(A/t_1A)\cong \mathcal E(A)/(t_1)$. Then, $t_2\in \mathcal E(A)$ defines a class $\bar{t}_2$ in $\mathcal E(A)/(t_1)\cong
\mathcal E(A/t_1A)$ and we set 
\begin{equation}
A/(t_1,t_2)A:=colim(A/t_1A\overset{\bar{t}_2}{\longleftarrow}A/t_1A \longrightarrow 0)
\end{equation} Again, $A/(t_1,t_2)A$ is a monoid and using \eqref{112}, we conclude that 
\begin{equation} \mathcal E(A/(t_1,t_2)A)\cong \mathcal E(A/t_1A)/(\bar{t}_2)
\cong \mathcal E(A)/(t_1,t_2)
\end{equation} where $(t_1,t_2)$ is the ideal in $\mathcal E(A)$ generated by $t_1$ and $t_2$. More generally, for each $2\leq i\leq n$, we set
\begin{equation}\label{114}
A/(t_1,...,t_i)A:=colim(A/(t_1,...,t_{i-1})A\overset{\bar{t}_i}{\longleftarrow} A/(t_1,...,t_{i-1})A \longrightarrow 0)
\end{equation} and note that $\mathcal E(A/(t_1,...,t_i)A)\cong
\mathcal E(A)/(t_1,...,t_i)$.

\medskip
\begin{thm}\label{prop17} Let $A$ be an object of $Comm(\mathbf C)$ and let $\{t_i\}_{i\in I}$ be a collection of elements $t_i\in \mathcal E(A)$ such that
the collection $\{Spec(A_{t_i})\longrightarrow Spec(A)\}_{i\in I}$ forms a Zariski open cover of $Spec(A)$. Then, there exists a finite subcollection $\{t_i\}_{i\in I'}$, $I'\subseteq I$ that is a partition of unity on $A$. 
\end{thm} 

\begin{proof} By definition, since $\{Spec(A_{t_i})\longrightarrow Spec(A)\}_{i\in I}$ forms a Zariski open cover of $Spec(A)$, there is a finite subcollection $I'=\{1,2,....,n\}\subseteq I$ such that the collection of functors $\{\_\_\otimes_AA_{t_i}:A-Mod\longrightarrow A_{t_i}-Mod\}_{i\in I'}$ is conservative. We will show that $\{t_1,t_2,...,t_n\}$ is a partition of unity on $A$. For any chosen $j\in \{1,2,...,n\}$, we have:
\begin{equation}
A/(t_1,...,t_j)A\otimes_AA_{t_j}
\cong colim(A\overset{t_j}{\longleftarrow}A\longrightarrow 0)\otimes_AA_{t_j}\otimes_AA/(t_1,...,t_{j-1})A =0
\end{equation} From \eqref{114}, it now follows that $A/(t_1,...,t_n)A\otimes_AA_{t_j}=0$ and  for each $1\leq j\leq n$. Since the collection of functors $\{\_\_\otimes_AA_{t_i}:A-Mod\longrightarrow A_{t_i}-Mod\}_{i\in I'}$ is conservative, it follows that $A/(t_1,...,t_n)A=0$. Hence, $\mathcal E(A/(t_1,...,t_n)A)=\mathcal E(A)/(t_1,...,t_n)=0$. Hence, $\{t_1,...,t_n\}$ forms a partition of unity on $A$. 

\end{proof}

\medskip

\medskip

\section{Noetherian monoids over $(\mathbf C,\otimes,1)$} 

\medskip 

\medskip

\medskip In this section, we will begin to describe the properties of  Noetherian monoids and Noetherian schemes
over $(\mathbf C,\otimes,1)$ (see Definition \ref{revD3.1}
and Definition \ref{revD3.6}). As in usual algebraic geometry, we show that being Noetherian is 
a local property for schemes over $(\mathbf C,\otimes,1)$. In Section 2, for any monoid $A$
and any $t\in\mathcal E(A)$, we have already described the
``quotient monoid'' $A/tA$. We will extend this definition further
to introduce, for any ideal $\mathscr I\subseteq \mathcal E(A)$, 
a ``quotient monoid'' $A/\mathscr I$. When $A$ is Noetherian, we
show that any quotient $A/\mathscr I$ is also Noetherian. These results will put in place the basic framework for construction
of closed subschemes of a Noetherian scheme $X$ over $(\mathbf C,\otimes,1)$, which will ultimately be done   in Section 5.  We start by
defining Noetherian monoids in $(\mathbf C,\otimes,1)$.

\begin{defn}\label{revD3.1}  Let $A$ be a commutative monoid object in $(\mathbf C,\otimes,1)$.

\medskip
(a) Let $\{M_i\}_{i\in I}$ be a  filtered inductive system of objects of $A-Mod$ connected by monomorphisms and let $M:=colim_{i\in I}M_i$. Then, an object
$N$ of $A-Mod$ will be said to be finitely generated if the canonical map
\begin{equation}
colim_{i\in I}Hom_{A-Mod}(N,M_i)\longrightarrow Hom_{A-Mod}(N,M)
\end{equation} is a bijection.

\medskip
(b)   The monoid $A$ will be said to be Noetherian if every subobject of $A$
in $A-Mod$ is finitely generated in $A-Mod$. 
\end{defn}

\medskip
We remark here that our notion of Noetherian in Definition \ref{revD3.1}
is different from the notion of Noetherian considered in \cite{0AB1}
which was further modified in \cite{1AB1}.

\medskip

\begin{lem}\label{2.1L} Let $f:A\longrightarrow B$ be an epimorphism of monoids in $(\mathbf C,\otimes,1)$. Then, we have an isomorphism $B\cong B\otimes_AB$. 
\end{lem}

\begin{proof}  Suppose that we have morphisms $f_1,f_2:B\longrightarrow C$ in $Comm(\mathbf C)$  such that
$f_1\circ f=f_2\circ f:A\longrightarrow C$. Since $f:A\longrightarrow B$ is an epimorphism in $Comm(\mathbf C)$, it follows that $f_1=f_2$. Hence, $B$ is the pushout of the diagram $B\overset{f}{\longleftarrow}A\overset{f}{\longrightarrow}B$ in $Comm(\mathbf C)$. Moreover, we know that  for any given morphisms  $g:D\longrightarrow D'$, $h:D\longrightarrow D''$ of objects in $Comm(\mathbf C)$, the following  is a pushout diagram in
$Comm(\mathbf C)$ (see, for instance, \cite[Lemma 2.3]{AB12})
\begin{equation}\label{rev3.1e}
\begin{CD}
D @>g>> D' \\
@VhVV @VVV \\
D'' @>>> D'\otimes_DD'' \\
\end{CD}
\end{equation} In particular, it follows from  \eqref{rev3.1e} that $B\otimes_AB$
is the pushout of the diagram $B\overset{f}{\longleftarrow}A\overset{f}{\longrightarrow}B$ in $Comm(\mathbf C)$.
Hence, we have $B\cong B\otimes_AB$. 
\end{proof}

\medskip
\begin{thm}\label{2.3P} Let $A$ be a commutative monoid object in $(
\mathbf C,\otimes,1)$ and let $f:A\longrightarrow B$ be a morphism of monoids inducing a Zariski open immersion of affine schemes. Then, if $A$ is Noetherian, so is $B$. 
\end{thm}

\begin{proof} From Definition \ref{zarop}, we know that $f:A
\longrightarrow B$ is an epimorphism of monoids. Hence, from Lemma \ref{2.1L} we know that $B\cong  B\otimes_AB$. It follows that, for any $B$-module $M$, we have
\begin{equation}\label{rev-33}M\otimes_AB\cong M\otimes_BB\otimes_AB\cong M\otimes_BB\cong M
\end{equation}  Let $J$ be a subobject of $B$ in $B-Mod$. We consider the following pullback square in $A-Mod$:
\begin{equation}\label{rev33}
\begin{CD}
I @>>> A \\
@VVV @VVV \\
J @>>> B \\
\end{CD}
\end{equation} Then, $I$ is a subobject of $A$ in 
$A-Mod$. Since $A$ is Noetherian, $I$ is a finitely generated object
of $A-Mod$. Since $B$ is a flat $A$-module,  the following is also a pullback square:
\begin{equation}\label{rev33xy}
\begin{CD}
I\otimes_AB @>>> A\otimes_AB\cong B \\
@VVV @V\cong VV \\
J\otimes_AB\cong J @>>> B\otimes_AB \cong B\\
\end{CD}
\end{equation} where the isomorphism $J\otimes_AB\cong J$ appearing in \eqref{rev33xy} follows from \eqref{rev-33}. Since the right vertical arrow in the pullback square \eqref{rev33xy} is an isomorphism, we have $J\cong I\otimes_AB$. Now, suppose that
$\{M_i\}_{i\in I}$ is a  filtered inductive system of objects in $B-Mod$ 
connected by monomorphisms and
let $M:=colim_{i\in I}M_i$. Then, it follows that
\begin{equation}\label{24}
\begin{array}{l}
colim_{i\in I} Hom_{B-Mod}(J,M_i) 
\cong colim_{i\in I}Hom_{B-Mod}(I\otimes_AB,M_i) \\
\cong colim_{i\in I}Hom_{A-Mod}(I,M_i) 
\cong Hom_{A-Mod}(I,M)\\ \cong Hom_{B-Mod}(J,M) \\
\end{array}
\end{equation} where the isomorphism $colim_{i\in I}Hom_A(I,M_i) 
\cong Hom_A(I,M)$ appearing in \eqref{24} follows from the fact
that $I$ is a finitely generated object of $A-Mod$.
Hence, $J$ is a finitely generated object of $B-Mod$. It follows  that $B$ is Noetherian. 
\end{proof} 

\medskip
\begin{thm}\label{2.6prp} Let $A$ be a Noetherian commutative monoid object in $(\mathbf C,\otimes,1)$. Then, $\mathcal E(A)$ is a Noetherian ring.
\end{thm}

\begin{proof} Suppose that $\mathcal E(A)$ is non-Noetherian. Then, there exists a sequence $\{t_i\}_{i\in \mathbb N}$ of elements of $\mathcal E(A)$ such that we have a strictly increasing chain of ideals: 
\begin{equation}\label{rrev3.7eq}
(t_1)\subsetneq (t_1,t_2)\subsetneq (t_1,t_2,t_3)\subsetneq \dots
\end{equation} in $\mathcal E(A)$ that does not stabilise.  Let $\mathscr I_i$ be the ideal generated
by the elements $\{t_1,t_2,...,t_i\}$.  For each $i\in \mathbb N$, we define:
\begin{equation}
I_i:=lim(A\longrightarrow A/t_1A\otimes_AA/t_2A\otimes_A\dots\otimes_AA/t_iA\longleftarrow 0)
\end{equation} It is clear that we have a  chain of subobjects of $A$ in $A-Mod$ as follows
\begin{equation}\label{214e}
I_1\overset{h_1}{\longrightarrow} I_2\overset{h_2}{\longrightarrow} I_3\overset{h_3}{\longrightarrow} \dots
\end{equation}
with each $I_i$ a subobject of $I_{i+1}$. Suppose that for some given $i_0\in \mathbb N$, we have
\begin{equation}
A/t_1A\otimes_AA/t_2A\otimes_A\dots\otimes_AA/t_{i_0}A
\cong A/(t_1,...,t_{i_0})A\end{equation} Then, it follows that:
\begin{equation}\label{215eqcn1}
\begin{array}{l}
A/t_1A\otimes_AA/t_2A\otimes_A\dots\otimes_AA/t_{i_0+1}A \\
\cong A/t_1A\otimes_AA/t_2A\otimes_A\dots\otimes_AA/t_{i_0}A\otimes_Acolim(A\overset{t_{i_0+1}}{\longleftarrow}A\longrightarrow 0)\\
\cong colim(A/t_1A\otimes_A\dots\otimes_AA/t_{i_0}A\overset{t_{i_0+1}}{\longleftarrow}A/t_1A\otimes_A\dots\otimes_AA/t_{i_0}A\longrightarrow 0) \\
\cong colim(A/(t_1,...,t_{i_0})A\overset{t_{i_0+1}}{\longleftarrow}A/(t_1,...,t_{i_0})A\longrightarrow 0)\\
\cong A/(t_1,....,t_{i_0},t_{i_0+1})A \\
\end{array}
\end{equation} Using induction, it follows from \eqref{215eqcn1} that
\begin{equation}\label{215eqcn}
A/t_1A\otimes_AA/t_2A\otimes_A\dots\otimes_AA/t_{i}A
\cong A/(t_1,...,t_{i})A\qquad \forall\textrm{ }i\in \mathbb N\end{equation}
 From \eqref{215eqcn}, it follows that
\begin{equation}\label{216e}
\begin{array}{l}
Hom_{A-Mod}(A,A/t_1A\otimes_AA/t_2A\otimes_A\dots\otimes_AA/t_iA)
\\ \cong Hom_{A-Mod}(A,A/(t_1,t_2,...,t_i)A) \\
 \cong Hom_{A/t_1A-Mod}(A/t_1A,A/(t_1,t_2,...,t_i)A)\\
\cong \dots \\
\cong Hom_{A/(t_1,t_2,...,t_{i-1})A-Mod}(A/(t_1,...,t_{i-1})A,A/(t_1,t_2,...,t_i)A) \\
\cong Hom_{A/(t_1,t_2,...,t_i)A-Mod}(A/(t_1,t_2,...,t_i)A,A/(t_1,t_2,...,t_i)A)\cong \mathcal E(A/(t_1,t_2,...,t_i)A) \\
\end{array}
\end{equation}  From the discussion preceding  Proposition \ref{prop17}, we know that
\begin{equation}\label{217e}
\mathcal E(A/(t_1,t_2,...,t_i)A)\cong \mathcal E(A)/(t_1,...,t_i)
\end{equation} where $(t_1,...,t_i)$ in \eqref{217e} denotes
the ideal in $\mathcal E(A)$ generated by $\{t_1,...,t_i\}$. From \eqref{216e} and \eqref{217e}, it follows that for any $i\in \mathbb N$,
\begin{equation}\label{312ng}
\begin{array}{l}
Hom_{A-Mod}(A,I_i)\cong Hom_{A-Mod}(A,lim(A\longrightarrow A/t_1A\otimes_AA/t_2A\otimes_A\dots\otimes_AA/t_iA\longleftarrow 0))\\
 \cong lim(Hom_{A-Mod}(A,A)\longrightarrow Hom_{A-Mod}(A,A/t_1A\otimes_AA/t_2A\otimes_A\dots\otimes_AA/t_iA)\longleftarrow 0) \\
 \cong lim(\mathcal E(A)\longrightarrow  \mathcal E(A)/(t_1,...,t_i)\longleftarrow 0) \cong (t_1,t_2,...,t_i)=\mathscr I_i\subseteq \mathcal E(A)\\ 
\end{array}
\end{equation} We now consider the chain  of subobjects 
$\{I_i\}_{i\in \mathbb N}$ described in \eqref{214e} and set
$I=colim_{i\in \mathbb N}I_i$. For any $i$, let $h'_i:I_i\longrightarrow I$ be the canonical morphism from $I_i$ to the colimit $I$. Since $A$ is Noetherian, $I$ is a finitely generated object of $A-Mod$. Hence,
\begin{equation}\label{218e}
Hom_{A-Mod}(I,I)\cong colim_{i\in \mathbb N}Hom_{A-Mod}(I,I_i)
\end{equation} In particular, it follows from \eqref{218e} that 
there exists $j\in \mathbb N$ such that for each $i\geq j$, there is a morphism $g_i:I\longrightarrow I_i$ such that 
$1=h'_i\circ g_i:I\overset{g_i}{\longrightarrow}I_i
\overset{h'_i}{\longrightarrow}I$. Since each $h_i:I_i\longrightarrow I_{i+1}$ in \eqref{214e} is a monomorphism, so is each morphism $h'_i:I_i\longrightarrow I$ to the colimit $I$. However, since $1=g_i\circ h_i'$ for any $i\geq j$, it follows that $h_i'$ is also an epimorphism and hence $h_i':I_i\longrightarrow I$ is an isomorphism for $i\geq j$ 
($A-Mod$ being an abelian category). It follows that each $h_i:I_i\longrightarrow I_{i+1}$, $i\geq j$ is an isomorphism. 
Combining with \eqref{312ng}, it follows therefore, that for each $i\geq j$, we have
\begin{equation}
\mathscr I_i=Hom_A(A,I_i)=Hom_A(A,I_{i+1})=\mathscr I_{i+1}
\end{equation} which is a contradiction. Hence, $\mathcal E(A)$
is Noetherian.

\end{proof}

\medskip

\begin{thm}\label{Noeth2} Let $B$ be a commutative monoid object in $(\mathbf C,\otimes,1)$ such that there exists a finite Zariski covering $\{Spec(A_i)\longrightarrow Spec(B)\}_{i\in I}$ with each $A_i$, $i\in I$ a Noetherian
monoid. Then, $B$ is Noetherian. 
\end{thm}

\begin{proof} We consider a  filtered inductive system of $B$-modules $\{M_k\}_{k\in K}$ connected by monomorphisms and  set
 $M:=colim_{k\in K}M_k$. Let $h_k:M_k\longrightarrow M$ denote the canonical morphisms. Since $B-Mod$ is an abelian category, we know that a morphism $J\longrightarrow B$ in $B-Mod$ defines a subobject of $B$ if and only if:
 \begin{equation}
0=lim(J\longrightarrow B\longleftarrow 0)
\end{equation}  Let $J$ be a subobject of $B$ in $B-Mod$. Since each $A_i$, $i\in I$ is 
a flat $B$-module, we note that
\begin{equation}\label{222}
0=lim(J\longrightarrow B\longleftarrow 0)\otimes_BA_i
\cong lim(J\otimes_BA_i\longrightarrow A_i\longleftarrow 0)
\end{equation} It follows from \eqref{222} that each $J\otimes_BA_i$, $i\in I$ is a subobject of $A_i$. Further, we note that for each $k\in K$, we have canonical morphisms
\begin{equation}
Hom_{B-Mod}(J,M_k)\longrightarrow Hom_{B-Mod}(J,M)
\end{equation} induced by $h_k:M_k\longrightarrow M$. Conversely, suppose that we have a morphism
$f:J\longrightarrow M$ in $B-Mod$. We consider the induced 
morphisms $f\otimes_BA_i:J\otimes_BA_i\longrightarrow M\otimes_BA_i$ for each $i\in I$. Since each $h_k:M_k\longrightarrow M$ is a monomorphism and $A_i$ is a flat $B$-module, it follows that $\{M_k\otimes_BA_i\}_{k\in K}$ is also a 
filtered system connected by monomorphisms for each $i\in I$.
Moreover, since each $A_i$ is Noetherian, we know that
\begin{equation}
Hom_{A_i-Mod}(J\otimes_BA_i,M\otimes_BA_i)\cong colim_{k\in K}Hom_{A_i-Mod}(J\otimes_BA_i,M_k\otimes_BA_i)
\end{equation} Since $I$ is finite and $K$ is filtered, it follows that there exists some
$k_0\in K$ such that each morphism $f\otimes_BA_i:J\otimes_BA_i
\longrightarrow M\otimes_BA_i$, $i\in I$ factors through $M_{k_0}\otimes_BA_i$, i.e., there exist  morphisms
$g_{l,i}:J\otimes_BA_i\longrightarrow M_l\otimes_BA_i$
$\forall$ $i\in I$, $l\geq k_0$ such that $(h_l\otimes_BA_i)\circ g_{l,i}=f\otimes_BA_i$. Then, for any $i$, $i'\in I$, it follows that \begin{equation}\label{323pl}
(h_l\otimes_BA_i\otimes_BA_{i'})\circ (g_{l,i}\otimes_BA_{i'})=f\otimes_BA_i\otimes_BA_{i'}= (h_l\otimes_BA_i\otimes_BA_{i'})\circ (g_{l,i'}\otimes_BA_{i})
\end{equation}  as morphisms from $J\otimes_BA_i\otimes_BA_{i'}$ to $M\otimes_BA_i\otimes_BA_{i'}$. Since each $h_l:M_l\longrightarrow M$ is a monomorphism and $A_i$, $A_{i'}$ are flat $B$-modules, 
\begin{equation}\label{324pl} (h_l\otimes_BA_i\otimes_BA_{i'}):
M_l\otimes_BA_i\otimes_BA_{i'}\longrightarrow M\otimes_BA_i\otimes_BA_{i'}
\end{equation} is a monomorphism. From \eqref{323pl} and \eqref{324pl}, it follows that, for any $i$, $i'\in I$, we have
\begin{equation}
g_{l,i}\otimes_BA_{i'}=g_{l,i'}\otimes_BA_{i}:J\otimes_BA_i\otimes_BA_{i'}
\longrightarrow M_l\otimes_BA_i\otimes_BA_{i'}
\end{equation} Hence, for each $l\geq k_0$, using \cite[Th\'{e}or\`{e}me 2.5]{Toen2}, \cite[Corollaire 2.11]{Toen2}, we have an induced morphism 
\begin{equation}\label{225e}
\begin{CD}
J= lim(\prod_{i\in I}J\otimes_BA_i \overset{\longrightarrow}{\underset{\longrightarrow}{}} \prod_{i,i'\in I} J\otimes_BA_i\otimes_BA_{i'}) \\
@Vg_lVV \\
M_l= lim(\prod_{i\in I}M_l\otimes_BA_i \overset{\longrightarrow}{\underset{\longrightarrow}{}} \prod_{i,i'\in I} M_l\otimes_BA_i\otimes_BA_{i'}) \\
\end{CD}
\end{equation} where the limits in \eqref{225e} are taken in
$B-Mod$.  It follows that the morphism $f:J\longrightarrow M$ factors through $M_l$ for each $l\geq k_0$. Hence, $Hom_{B-Mod}(J,M)\cong colim_{k\in K}Hom_{B-Mod}(J,M_k)$ and $B$ is Noetherian. 

\end{proof} 

\medskip We are now ready to define a Noetherian scheme over $(\mathbf C,\otimes,1)$. 

\begin{defn}\label{revD3.6} Let $X$ be a scheme over $(\mathbf C,\otimes,1)$. Then, $X$ is said to be Noetherian if, for any Zariski open immersion
$U\longrightarrow X$ with $U=Spec(A)$ affine, $A$ is a Noetherian
monoid. 
\end{defn}

\medskip
\begin{thm} Let $X$ be a scheme over $(\mathbf C,\otimes,1)$ such that
there exists a  covering $\{U_i\longrightarrow X\}_{i\in I}$
with $U_i=Spec(A_i)$ affine such that each $A_i$ is Noetherian. 
Then, $X$ is a Noetherian scheme. 
\end{thm}

\begin{proof} We choose any Zariski open immersion $U\longrightarrow X$
with $U=Spec(B)$ affine. Then, we consider the pullback squares
\begin{equation}
\begin{CD}
X_i @>>> U =Spec(B)\\
@VVV @VVV \\
U_i=Spec(A_i) @>>> X \\
\end{CD}
\end{equation} We choose an affine covering $\{U_{ij}=Spec(A_{ij})\longrightarrow X_i\}_{i\in J_i}$ for
each $i\in I$. Since each $Spec(A_{ij})$, $j\in J_i$ admits a Zariski open
immersion into $Spec(A_i)$, it follows from Proposition \ref{2.3P}
that each $A_{ij}$ is a Noetherian monoid. Further, it is clear
that the collection $\{Spec(A_{ij})\longrightarrow Spec(B)\}_{j\in J_i,i\in I}$ (and hence a finite subcollection thereof) is a Zariski covering of $Spec(B)$. It now follows from Proposition \ref{Noeth2}
that $B$ is a Noetherian monoid. This proves the result. 
\end{proof}

\medskip 

Given a monoid $A$ and any $t\in \mathcal E(A)$, the construction of the ``quotient monoid'' $A/tA$ has been introduced in \eqref{110}. We will now generalise this construction. Let $\mathscr I\subseteq \mathcal E(A)$ be an ideal. For each $t\in \mathscr I$, we can consider the quotient $A/tA$ as defined in \eqref{110}. We now define:
\begin{equation}\label{2.1e}
A/\mathscr I:=colim\{A\longrightarrow A/tA\}_{t\in \mathscr I}
\end{equation} the colimit in \eqref{2.1e} being taken in the 
category $Comm(\mathbf C)$. By definition, it follows that $A/\mathscr I$ is a commutative monoid in $(\mathbf C,\otimes,1)$ and that the canonical
morphism $A\longrightarrow A/\mathscr I$ is a morphism of commutative
monoids. 

\medskip
\begin{lem}\label{L2.01L} Let $A$ be a commutative monoid object in $(\mathbf C,\otimes,1)$ and let $t\in \mathcal E(A)$. Then, the canonical map
$A\longrightarrow A/tA$ is an epimorphism of monoids, i.e., for any monoid
$B$, the induced morphism
\begin{equation}
Hom_{Comm(\mathbf C)}(A/tA,B)\longrightarrow Hom_{Comm(\mathbf C)}(A,B)
\end{equation} is an injection. 
\end{lem}

\begin{proof} Let us denote by $p$ the canonical morphism $p:A\longrightarrow A/tA=colim(A\overset{t}{\longleftarrow}A
\longrightarrow 0)$. Hence, $p\circ t=0$.   Let $f,g:A/tA\longrightarrow B$ be  morphisms of monoids such that $f\circ p=g\circ p$. Then, the morphism $f\circ p=g\circ p:A\longrightarrow B$ makes $B$ an 
$A$-algebra and $f$, $g$ are maps of $A$-modules. Then, we have 
the following commutative diagram in $A-Mod$:
\begin{equation}
\begin{CD}
A @<t<< A @>>> 0 \\
@VVf\circ p=g\circ pV @VVf\circ p\circ tV
@VVV \\
B @<1<< B @>1>> B \\ \end{CD}
\end{equation} It follows that the morphism $f\circ p=g\circ p:A
\longrightarrow B$ in $A-Mod$ must factor uniquely through
the colimit $A/tA=colim(A\overset{t}{\longleftarrow}A\longrightarrow 0)$. Hence, $f=g$ and the result follows. 

\end{proof}

\medskip
\begin{thm}\label{2.2pr}(a) Let $A$ be a commutative monoid object in $(\mathbf C,\otimes,1)$ and let $\mathscr I=(t)$ be a principal ideal in $\mathcal E(A)$ generated by a given $t\in \mathcal E(A)$. Then, $A/tA\cong A/\mathscr I$. 

\medskip
(b) Let $A$ be a monoid object in $(\mathbf C,\otimes,1)$ and let
 $\mathscr I\subseteq \mathcal E(A)$ be a given ideal. Then, the canonical morphism $A\longrightarrow A/\mathscr I$ is an epimorphism of monoids, i.e., for any monoid
$B$, the induced morphism
\begin{equation}
Hom_{Comm(\mathbf C)}(A/\mathscr I,B)\longrightarrow Hom_{Comm(\mathbf C)}(A,B)
\end{equation} is an injection. 
\end{thm}

\begin{proof} (a) By definition, we know that $A/\mathscr I$ is given by the
colimit
\begin{equation}
A/\mathscr I:=colim\{p_x:A\longrightarrow A/xA\}_{x\in \mathscr I}
\end{equation} in $Comm(\mathbf C)$.  Suppose $B$ is a monoid such that there are
morphisms $q_x:A/xA\longrightarrow B$ of monoids such that
$q_x\circ p_x=q_y\circ p_y$ for all $x$, $y\in \mathscr I$. For any 
element $x$ in the principal ideal $\mathscr I$, there exists a
natural morphism 
$p_{x/t}:A/xA\longrightarrow A/tA$ of monoids such that $p_t=p_{x/t}\circ p_x$. Then, we note that:
\begin{equation}
q_x\circ p_x=q_t\circ p_t=q_t\circ p_{x/t}\circ p_x
\end{equation} From Lemma \ref{L2.01L}, it now follows that
$q_x=q_t\circ p_{x/t}$, i.e., each of the morphisms $q_x$ factors
through $q_t$. It follows that $A/tA$ is the colimit $colim\{p_x:A\longrightarrow A/xA\}_{x\in \mathscr I}=A/\mathscr I$
in $Comm(\mathbf C)$. 

\medskip
(b) Let $p:A\longrightarrow A/\mathscr I$ be the canonical morphism and let $f,g:A/\mathscr I\longrightarrow B$ be morphisms of monoids
such that $f\circ p=g\circ p$. For each $t\in \mathscr I$, the morphism $p:A\longrightarrow A/\mathscr I$ factors through the canonical morphism $p_t:A\longrightarrow A/tA$ and let $p_t'$ denote the canonical morphism $p_t':A/tA\longrightarrow A/\mathscr I$ to the colimit $A/\mathscr I$. We note that
\begin{equation}
(f\circ p_t')\circ p_t=f\circ p=g\circ p=(g\circ p_t')\circ p_t
\end{equation} From Lemma \ref{L2.01L}, it now follows that 
$f\circ p_t'=g\circ p_t'$. Hence, the collection of morphisms
$\{f\circ p_t'=g\circ p_t'\}_{t\in \mathscr I}$ factors uniquely through the 
colimit $A/\mathscr I$. This proves the result. 

\end{proof}

\medskip
Finally, we show that if $A$ is a Noetherian monoid and $\mathscr I\in \mathcal E(A)$ is an ideal in $\mathcal A$, the 
monoid $A/\mathscr I$ is also Noetherian. We start with the 
following result.

\medskip

\begin{thm}\label{2.8ppr} Let $A$ be a Noetherian commutative monoid and let  $t\in \mathcal E(A)$. Then, $A/tA$ is a Noetherian monoid.
\end{thm} 

\begin{proof} We consider a subobject $ J\longrightarrow A/tA$ in $A/tA-Mod$. We then form the following pullback diagram 
in $A-Mod$:
\begin{equation}\label{murfet}
\begin{CD}
 I @>>> A \\
@Vp'VV @VpVV \\
 J @>>> A/tA \\
\end{CD}
\end{equation} It is clear that $I$ is a subobject of
$A$ in $A-Mod$ and hence finitely generated.  By definition, we know that
\begin{equation}
A/tA:=colim(A\overset{t}{\longleftarrow}A\longrightarrow 0) 
\end{equation} Since the morphism $A\longrightarrow 0$ is 
an epimorphism in $A-Mod$, it follows that the canonical morphism
$p:A\longrightarrow A/tA$ is an epimorphism in $A-Mod$. Since
$A-Mod$ is an abelian category, we know that epimorphisms
are stable under pullback and hence the morphism $p':I\longrightarrow J$ 
in \eqref{murfet} is an epimorphism in $A-Mod$. Hence, $Im(p')=J$. Now, we set
\begin{equation}
K:=Ker(p')=lim(I\overset{p'}{\longrightarrow}J\longleftarrow 0)
\end{equation} and consider the coimage
\begin{equation}
Coim(p'):=colim(I\overset{i}{\longleftarrow} K \longrightarrow 0)
\end{equation} Since $A-Mod$ is an abelian category, the image and the coimage of $p'$ coincide and we have
\begin{equation}\label{coimage}
J=Im(p')=Coim(p')=colim(I\overset{i}{\longleftarrow} K \longrightarrow 0)
\end{equation} 

\medskip Next, we suppose that we have a filtered inductive system
of objects $\{M_l\}$, $l\in L$ connected by monomorphisms in $A/tA-Mod$ and set 
$M:=colim_{l\in L}\textrm{ }M_l$. Let $h_l:M_l\longrightarrow M$, $l\in L$ denote the canonical morphisms. It is clear that we have a 
canonical morphism
\begin{equation}
colim_{l\in L}Hom_{A/tA-Mod}(J,M_l)\longrightarrow Hom_{A/tA-Mod}(J,M)
\end{equation} We now choose any morphism $f:J\longrightarrow 
M$ in $A/tA-Mod$. Let $f_I:I\longrightarrow J$   be the canonical morphism in
$A-Mod$ induced by \eqref{coimage}.  We  consider the morphism $f\circ f_I:I\longrightarrow M$  in
$A-Mod$. Since  $I$ is finitely generated in $A-Mod$ and the system
$L$ is filtered, there exists some $l_0\in L$ such that the morphism 
$f\circ f_I$   factors through $M_{l_0}$, i.e., there exists $g:I\longrightarrow M_{l_0}$ such that $h_{l_0}\circ g=f\circ f_I$. Further, we have
\begin{equation} \label{341cx}
h_{l_0}\circ g\circ i=f\circ f_I\circ i=0:K\longrightarrow M
\end{equation} Since $h_{l_0}:M_{l_0}\longrightarrow M$ is a monomorphism, it follows from \eqref{341cx} that $g\circ i=0:
K\longrightarrow M_{l_0}$.  Since $J$ is equal to the colimit in \eqref{coimage}, it follows that the morphism $f:J\longrightarrow M$ factors through $M_{l_0}$ in $A-Mod$. 

\medskip
Finally, since the canonical morphism $p:A\longrightarrow A/tA$ is also 
an epimorphism in the category of monoids (as shown in Proposition \ref{2.2pr}), it follows from Lemma \ref{2.1L} that
\begin{equation}
J\otimes_AA/tA\cong (J\otimes_{A/tA}A/tA)\otimes_AA/tA\cong J\otimes_{A/tA}(A/tA\otimes_AA/tA) \cong J\otimes_{A/tA}A/tA\cong J
\end{equation} Hence, for any object $N$ in $A/tA-Mod$,
\begin{equation}
Hom_{A-Mod}(J,N)\cong Hom_{A/tA-Mod}(J\otimes_AA/tA,N)\cong Hom_{A/tA-Mod}(J,N)
\end{equation} Now, it follows that since the morphism $f:J\longrightarrow M$ factors through $M_{l_0}$ in
$A-Mod$, it actually factors through $M_{l_0}$ in $A/tA-Mod$. Hence, $J$ is finitely generated in $A/tA-Mod$. This proves the result. 

\end{proof}

\medskip
\begin{thm}\label{noethlast} Let $A$ be a Noetherian commutative monoid and let $\mathscr I\subseteq \mathcal E(A)$ be an ideal in $\mathcal E(A)$. Then, 
$A/\mathscr I$ is a Noetherian monoid. 
\end{thm}

\begin{proof} Since $A$ is a Noetherian monoid, it follows from
Proposition \ref{2.6prp} that $\mathcal E(A)$ is actually a Noetherian
ring. Hence, we may suppose that the ideal $\mathscr I$ is generated
by a finite set $\{t_1,...,t_k\}\subseteq \mathcal E(A)$. As in \eqref{114}, we set, for $2\leq i\leq k$:
\begin{equation}\label{114c}
A/(t_1,...,t_i)A:=colim(A/(t_1,...,t_{i-1})A\overset{t_i}{\longleftarrow} A/(t_1,...,t_{i-1})A \longrightarrow 0)
\end{equation} From Proposition \ref{2.8ppr}, we know that
$A/t_1A$ is Noetherian. From the recursive definition in 
\eqref{114c}, it follows that each $A/(t_1,...,t_i)A$ is
Noetherian. Further, from \eqref{215eqcn}, we know that 
\begin{equation}\label{345cn} A/t_1A\otimes_AA/t_2A\otimes_A\dots\otimes_AA/t_iA\cong A/(t_1,t_2,...,t_i)A
\end{equation} As in \eqref{rev3.1e} in the proof of Lemma \ref{2.1L}, for the finite collection $A/t_1A$, $A/t_2A$,...,$A/t_kA$ of $A$-algebras, we know that
\begin{equation}\label{236}
C:=Colim_{1\leq i\leq k}\{A\longrightarrow A/t_iA\}\cong A/t_1A\otimes_AA/t_2A\otimes_A\dots\otimes_AA/t_kA
\end{equation} where the colimit in \eqref{236} is taken
in the category of monoids. For any $1\leq i\leq k$, we let $e_i:A/t_iA\longrightarrow C$
be the canonical morphism from $A/t_iA$ to the colimit $C$ described
in \eqref{236}. The induced morphism from $A$ to $C$ will be denoted
by $e$.  Further, for any $t\in \mathscr I$, let $p_t:A\longrightarrow A/tA$ denote
the canonical epimorphism described in Lemma \ref{L2.01L}.

\medskip 
Since 
$\mathscr I$ is generated by $\{t_1,...,t_k\}$, for any $t\in \mathscr I$, we can choose $s_i\in \mathcal E(A)$, $1\leq i\leq k$ such that
$t=\sum_{i=1}^kt_is_i$. We note that:
\begin{equation}
e\circ t=e\circ \sum_{i=1}^kt_is_i=\sum_{i=1}^ke\circ t_is_i=\sum_{i=1}^k e_i\circ (p_{t_i}\circ t_i)\circ s_i=0
\end{equation} Hence, there exists a unique morphism $e_t:A/tA\longrightarrow C$, $e_t\circ p_t=e$,   in $A-Mod$ from the 
colimit $A/tA:=colim\{A\overset{t}{\longleftarrow}A\longrightarrow 0\}$ to $C$ that may be easily shown to be a morphism
of monoids. Hence, the morphism $e:A\longrightarrow C$ factors
through $A/tA$ for any $t\in \mathscr I$ in the category of
monoids. It follows that
\begin{equation}\label{239}
A/\mathscr I=\underset{t\in \mathscr I}{colim}\{A\longrightarrow A/tA\}
\cong \underset{1\leq i\leq k}{colim}\{A\longrightarrow A/t_iA\}\cong 
A/t_1A\otimes_AA/t_2A\otimes_A\dots\otimes_AA/t_kA
\end{equation} Combining \eqref{239} with \eqref{345cn} and the fact that
$A/(t_1,t_2,...,t_k)A$ is Noetherian, it follows that $A/\mathscr I$ is Noetherian.
\end{proof} 

\medskip
\begin{rem}\emph{The functor $\mathcal E:Comm(\mathbf C)
\longrightarrow Rings$ plays a key role in our constructions 
above and in the rest of the paper. As such, we conclude
this section by briefly summarizing some of the properties
of this functor $\mathcal E$ that we have proved above: } 

\medskip
\emph{(a) The functor $\mathcal E$ preserves localizations, i.e.,
if $A\in Comm(\mathbf C)$ and 
if $S\subseteq \mathcal E(A)$ is a mutliplicatively closed set, 
then $\mathcal E(A_S)=\mathcal E(A)_S$. } 

\emph{(b) If $\{A_i\}_{i\in I}$ is a filtered system of objects
in $Comm(\mathbf C)$, we have $\mathcal E(colim_{i\in I}A_i)
=colim_{i\in I}\mathcal E(A_i)$.  }

\emph{(c) If $A\in Comm(\mathbf C)$ is a Noetherian commutative monoid
object, $\mathcal E(A)$ is a Noetherian ring.}

\emph{(d) If $\mathscr I\subseteq \mathcal E(A)$ is a finitely
generated ideal, we have $\mathcal E(A/\mathscr I)=
\mathcal E(A)/\mathscr I$. }

\end{rem}

\medskip

\medskip

\section{Integral schemes and function field}

\medskip

\medskip
In this section, we will introduce the definition and describe  the properties of integral 
schemes over $(\mathbf C,\otimes,1)$, in addition to reduced and irreducible schemes. 
In particular, we will show that an integral scheme is both reduced and irreducible. For our
purposes, we will need to consider a second notion of integrality for monoids (see Definition 
\ref{paoh1}) that we shall refer to as ``weak integrality''. We will show that a reduced
and irreducible scheme over $(\mathbf C,\otimes,1)$ is weakly integral. Moreover, for any integral scheme $X$ over $(\mathbf C,\otimes,1)$, we will construct a field $k(X)$ that is the appropriate analogue of the function field of an ordinary integral scheme and show that this association 
is functorial with respect to dominant morphisms.  Further, we will verify that the field $k(X)$ associated
to an integral scheme $X$ over $(\mathbf C,\otimes,1)$ is  completely determined
by any open subscheme $U$ of $X$. We mention here that when $A$ is 
``weakly integral'' in the sense of Definition \ref{paoh1} below
and also  ``Noetherian'' in the sense of \cite{1AB1} (which 
is different from the notion of Noetherian in Definition \ref{revD3.1}), we have
constructed in \cite{1AB1}  a monoid object $K(A)\in \mathbf C$ with some ``field
like properties''.  We start by presenting
the following two definitions. 

\medskip
\begin{defn}\label{paoh1} (Weakly integral monoids) Let $A$ be a commutative monoid object in 
$(\mathbf C,\otimes,1)$. Then, we will say that $A$ is weakly integral if 
$\mathcal E(A)$ is an integral domain. 
\end{defn}

\medskip
\begin{defn}\label{Df4.2n} (Integral monoids) Let $A$ be a weakly integral monoid object in $(\mathbf C,\otimes,1)$. We will say that $A$ is an integral monoid if it satisfies the following two
conditions:

\medskip
(1) For any element $s\in \mathcal E(A)=Hom_{A-Mod}(A,A)$ such that 
$s\ne 0$, the morphism
$s:A\longrightarrow A$ is a monomorphism in $A-Mod$. 

\medskip
(2) Let $S\subseteq \mathcal E(A)$ be the multiplicatively closed subset of
all non-zero elements in the integral domain $\mathcal E(A)$ and let
$K:=A_S$. Then, $K$ has no proper subobjects in $K-Mod$, i.e., any monomorphism
$J\longrightarrow K$ in $K-Mod$ with $J\ne 0$ is an isomorphism. 
\end{defn}

\medskip

\begin{thm}\label{after} Let $A$ be an integral monoid object of $(\mathbf C,\otimes,1)$ and let $S\subseteq \mathcal E(A)$ be a  multiplicatively closed subset such that $0\notin S$. Then, the canonical
morphism $i_S:A\longrightarrow A_S$ is a monomorphism in
$A-Mod$. 
\end{thm}

\begin{proof} Since $A$ is an integral monoid, we know that for any
$s\in S\subseteq \mathcal E(A)$, the morphism $s:A\longrightarrow A$ (and hence any 
$s^i:A\longrightarrow A$) is a monomorphism in $A-Mod$. We now have
the following morphism of filtered inductive systems in $A-Mod$:
\begin{equation}\label{tlif1}
\begin{CD}
A @>1>> A @>1>> A @>1>> \dots \\
@V1VV @VsVV @Vs^2VV \\
A @>s>> A @>s>> A @>s>> \dots \\
\end{CD}
\end{equation} with each vertical map in \eqref{tlif1} a monomorphism. Hence, the induced morphism on filtered colimits 
$i_s:A\longrightarrow A_s$ of the horizontal rows in \eqref{tlif1} is also a monomorphism. Again, it follows
that the   filtered colimit of monomorphisms $i_s:A\longrightarrow A_s$, $s\in S$, 
\begin{equation}
i_S:A\longrightarrow A_S=\underset{s\in S}{colim}\textrm{ }A_s
\end{equation} is a monomorphism.

\end{proof}

\medskip
By definition, we know that an integral monoid $A$ is also weakly integral. We will now
show that if $A$ is an integral monoid and $Spec(B)\longrightarrow Spec(A)$ is a 
Zariski open immersion of affine schemes, the monoid $B$ is weakly  integral.

\medskip
\begin{thm}\label{prp3201} Let $A$ be an integral monoid object in $(\mathbf C,\otimes,1)$ and let $f:A\longrightarrow B$ be a morphism of commutative monoids inducing a Zariski open immersion
of affine schemes. Then, $B$ is a weakly integral monoid.  
\end{thm}

\begin{proof} For the integral monoid $A$, we let $S\subseteq \mathcal E(A)$ be the multiplicatively closed set of all non zero elements 
in $\mathcal E(A)$ and let $K:=A_S$. Then, since $A$ is integral, it follows
from  Definition \ref{Df4.2n} that $K$ has no nonzero proper subobjects in $K-Mod$.

\medskip
We now consider the following pushout square in the category 
of monoids:
\begin{equation}
\begin{CD}
A @>i_A>> K \\
@VfVV @Vf_KVV \\
B @>i_B>> L=K\otimes_AB \\
\end{CD}
\end{equation} Since $f:A\longrightarrow B$ induces a 
Zariski open immersion, it follows that $f_K:K\longrightarrow L$
also induces a Zariski open immersion. We choose any $t\in \mathcal 
E(L)$ and consider $Ker(t)$, which is a subobject of $L$ in $L-Mod$. However, from the proof of Proposition \ref{2.3P}, we know that since $f_K:K\longrightarrow L$ induces a Zariski open immersion, every subobject of $L$ in $L-Mod$ is extended from a subobject
of $K$ in $K-Mod$. Since $K$ has no nonzero proper subobjects in
$K-Mod$, it follows that $Ker(t)=0$ or $Ker(t)=L$, i.e., if 
$t\ne 0$, then $Ker(t)=0$. Now, suppose that there exists an element
$t'\in \mathcal E(L)$ such that $t\circ t'=0$. Then, if $t\ne 0$,
\begin{equation}\label{47rc}
t\circ t'=0 \qquad \Rightarrow \qquad Im(t')\subseteq Ker(t)=0
\end{equation} where we note that the image $Im(t')$ exists as
a subobject of $L$ in the abelian category $L-Mod$. It follows from \eqref{47rc} that $t'=0$, i.e., $\mathcal E(L)$ is
an integral domain. 

\medskip
Further,  it follows from Proposition \ref{after} that the canonical morphism $i_A:A\longrightarrow K=A_S$ is a 
monomorphism in $A-Mod$. Since $B$ is a flat $A$-module, 
it follows that $i_B=i_A\otimes_AB:B\longrightarrow K\otimes_AB=L$ is a monomorphism in $B-Mod$. Finally, suppose that $s$, $s'\in \mathcal E(B)$ are morphisms such that $s\circ s'=0$. Consider the extensions $s\otimes_BL$, $s'\otimes_BL$ of $s$, $s'$ resp. to $\mathcal E(L)$. Since $\mathcal E(L)$ is an integral
domain, it follows that
\begin{equation}
(s\otimes_BL)\circ (s'\otimes_BL)=((s\circ s')\otimes_BL)=0 \qquad \Rightarrow \mbox{
$(s\otimes_BL)=0$ or $(s'\otimes_BL)=0$}
\end{equation} For sake of definiteness, we assume $s\otimes_BL=0$. Then, we note that, under the isomorphism $Hom_{B-Mod}(B,L)
\cong Hom_{L-Mod}(L,L)$, the morphism 
$s\otimes_BL\in Hom_{L-Mod}(L,L)$ corresponds to
\begin{equation}
(s\otimes_BL)\circ i_B=i_B\circ s: B\overset{s}{\longrightarrow} B 
\overset{i_B}{\longrightarrow} L 
\end{equation} in $Hom_{B-Mod}(B,L)$. Hence, $i_B\circ s=0$.  Since $i_B$ is a monomorphism in $B-Mod$, it follows that $s:B
\longrightarrow B$ is zero. Hence, $\mathcal E(B)$ is an integral domain and $B$ is a
weakly integral monoid. 

\end{proof}

\medskip
\begin{defn}\label{revD4.3} Let $X$ be a scheme over $(\mathbf C,\otimes,1)$. We will
say that $X$ is an integral scheme if, given any Zariski open immersion $Y\longrightarrow X$
with $Y=Spec(A)$ affine, $A$ is an integral monoid in $(\mathbf C,\otimes,1)$. 

\medskip We will say that $X$ is weakly integral (resp. reduced) if given any Zariski open immersion $Y\longrightarrow X$ with $Y=Spec(A)$ affine, $\mathcal E(A)$ is an integral domain (resp. a reduced ring). 

\medskip We will say that a scheme $X$ is irreducible if, given any Zariski open immersions, $U\longrightarrow X$,
$V\longrightarrow X$ with $U$ and $V$ nontrivial, the fibre product $U\times_XV$ is a nontrivial. 
\end{defn} 

\medskip
\begin{thm}\label{p340} Let $X$ be a scheme over $(\mathbf C,\otimes,1)$ that is both reduced and irreducible. Then, 
$X$ is a weakly  integral scheme. 
\end{thm}

\begin{proof} Let $i:U\longrightarrow X$ be a Zariski open immersion with $U=Spec(A)$ affine. Suppose that we  choose
$s$, $t\in \mathcal E(A)$ such that  $s\ne 0$ and $t\ne 0$. From Corollary \ref{cor1},  it follows that
$\mathcal E(A_s)=\mathcal E(A)_{s}$. Since $\mathcal E(A)$ is reduced, $\mathcal E(A)_s\ne 0$ and hence
$A_s\ne 0$. Similarly, $A_t\ne 0$. 

\medskip 
From Proposition \ref{openimmer}, we know that 
the compositions $Spec(A_s)\overset{i_t}{\longrightarrow} Spec(A)\overset{i}{\longrightarrow} X$ and $Spec(A_t)\overset{i_s}{\longrightarrow} Spec(A)\overset{i}{\longrightarrow} X$ are Zariski immersions. Since $X$ is irreducible and both $Spec(A_s)$ and $Spec(A_t)$ are non trivial, it follows that
the fibre product $Spec(A_s)\times_XSpec(A_t)$ is non trivial. 

\medskip Further, since a Zariski open immersion is a monomorphism in the category of schemes over $(\mathbf C,\otimes,1)$, it
follows that, given morphisms $f_s:Y\longrightarrow Spec(A_s)$, $f_t:Y\longrightarrow Spec(A_t)$ of schemes such that
$i\circ (i_s\circ f_s)=i\circ (i_t\circ f_t)$, we must have $i_s\circ f_s=i_t\circ f_t$. Hence, we have an isomorphism 
of fibre products:
\begin{equation}\label{390e}
Spec(A_s)\times_XSpec(A_t)\overset{\cong}{\longrightarrow}Spec(A_s)\times_{Spec(A)}Spec(A_t)=Spec(A_{st})=Spec(A_s\otimes_AA_t)
\end{equation} It follows that $Spec(A_{st})$ is non trivial. Hence, $A_{st}\ne 0$ and therefore
$st\ne 0$. Hence, $\mathcal E(A)$ is an integral domain and $A$ is a weakly
integral monoid. This proves the result. 

\end{proof}

\medskip
We will now  prove a partial converse to Proposition \ref{p340}. 
From Definition \ref{revD4.3}, it is clear that an integral scheme is
always reduced. We start by showing that if $A$ is an integral monoid, then
$Spec(A)$ is irreducible. 

\medskip
\begin{thm}\label{p441} Let $A$ be an integral monoid object in $(\mathbf C,\otimes,1)$. Then,
$Spec(A)$ is an irreducible scheme.
\end{thm}

\begin{proof} It suffices to show that if $U\longrightarrow Spec(A)$, $V\longrightarrow Spec(A)$ are Zariski open immersions
from affine schemes $U$ and $V$, the fibre product $U\times_{Spec(A)}V$ is non trivial. Suppose, therefore that $f:A\longrightarrow B$, $g:A\longrightarrow C$, $B\ne 0$, $C\ne 0$, are morphisms of monoids inducing Zariski
open immersions of affine schemes such that $B\otimes_AC=0$. We let
$S\subseteq \mathcal E(A)$ be the multiplicatively closed set of 
all nonzero elements of $\mathcal E(A)$ and we set $K:=A_S$. We now 
consider the following pushout squares in $Comm(\mathbf C)$:
\begin{equation}
\begin{array}{ll}
\begin{CD} 
A @>i>> K \\ @VfVV @Vf_KVV \\ B @>i_B>> B_K=B\otimes_AK \\ 
\end{CD} \quad & \quad  
\begin{CD} 
A @>i>> K \\ @VgVV @Vg_KVV \\ C @>i_C>> C_K=C\otimes_AK \\ 
\end{CD} 
\end{array}
\end{equation} Since $A$ is an integral monoid, it follows as in the
proof of Proposition \ref{after} that any morphism $s\in Hom_{A-Mod}(A,A)=\mathcal E(A)$,
$s\ne 0$ is a monomorphism in $A-Mod$. Since $B$ is flat over $A$,
it follows that
\begin{equation}\begin{CD}
\mathcal E(f)(s)=s\otimes_AB:B=A\otimes_AB@>s\otimes_AB>> 
A\otimes_AB=B \\
\end{CD} \end{equation} is a monomorphism in $B-Mod$. Hence, $\mathcal E(f)(s)
\ne 0$ and therefore $\mathcal E(f)(S)\subseteq \mathcal E(B)$ is a multiplicatively closed
subset of $\mathcal E(B)$ containing $1_B$ and not containing $0$. It
follows from the definition of localisation in  \eqref{vprim} that
\begin{equation}
B_K=B\otimes_AK=B_{\mathcal E(f)(S)}
\end{equation} 
Since $A$ is an integral monoid, it follows from Proposition \ref{after} that 
the canonical morphism $i:A\longrightarrow K=A_S$ is a monomorphism. Since $f:A\longrightarrow B$ induces a Zariski open
immersion, $B$ is a flat $A$-module. Consequently, $i_B=i\otimes_AB:
A\otimes_AB=B\longrightarrow K\otimes_AB=B_K=B_{\mathcal E(f)(S)}$ is a monomorphism in $B-Mod$ and hence $B_K\ne 0$. Similarly, $C_K\ne 0$. 

\medskip
Now, suppose that $f_K=0$. Then, $i_B\circ f=f_K\circ i=0$. Hence, 
\begin{equation} \label{413ef}
\begin{CD}
0=(i_B\circ f)\otimes_AB: B@>f\otimes_AB>> B\otimes_AB=B
@>i_B\otimes_AB>> B_K\otimes_AB=B\otimes_AB\otimes_AK=B_K
\end{CD}
\end{equation} where the equality $B\otimes_AB=B$ in \eqref{413ef}
follows from Lemma \ref{2.1L}. From \eqref{413ef}, we have
$0=(i_B\circ f)\otimes_AB=i_B:B\longrightarrow B_K$ which contradicts the fact that $i_B$ is a monomorphism in $B-Mod$. 
Hence, we must have $f_K\ne 0$. 

\medskip
We now consider the following kernels:
\begin{equation}\label{kernels1}
T_A:=lim_{A-Mod}(K\overset{f_K}{\longrightarrow}B_K
\longleftarrow 0) \qquad T_K:=lim_{K-Mod}(K\overset{f_K}{\longrightarrow}B_K
\longleftarrow 0)
\end{equation} where the first limit in \eqref{kernels1} is taken
in the category $A-Mod$ and the second is taken in $K-Mod$. Since $A$ is an integral monoid, 
we know that $K$ has no non-zero proper subobjects in $K-Mod$.  Then, since $f_K\ne 0$ as shown above,
the subobject $T_K$ of $K$ in $K-Mod$ must be zero. 

\medskip Further, as mentioned in Section 2, the canonical morphism 
$i:A\longrightarrow A_S=K$ defined by the localisation is an epimorphism in 
$Comm(\mathbf C)$. Using Lemma \ref{2.1L}, we have $K\otimes_AK=K$. Since 
$K=A_S$ is a flat $A$-module, it now follows that
\begin{equation}\label{zeros1}
T_A\otimes_AK=lim_{K-Mod}(K\otimes_AK=K\overset{f_K}
{\longrightarrow}B_K\otimes_AK=B\otimes_AK\otimes_AK=B_K
\longleftarrow 0)=T_K=0
\end{equation}   However, we also know that
\begin{equation}\label{zeros2}
Hom_{A-Mod}(T_A,K)\cong Hom_{K-Mod}(T_A\otimes_AK,K)=Hom_{K-Mod}(T_K,K)=0
\end{equation} Since $T_A$ is a subobject of $K$ in $A-Mod$, 
it follows from \eqref{zeros2} that $T_A=0$. Combining with
\eqref{kernels1}, it follows that $f_K:K\longrightarrow B_K$
is a monomorphism in $A-Mod$. Since $C$ is a flat $A$-module, it 
follows that 
\begin{equation}
f_K\otimes_AC:C_K=K\otimes_AC\longrightarrow B_K\otimes_AC=(B
\otimes_AC)\otimes_AK=0
\end{equation} is a monomorphism in $C-Mod$. Hence, $C_K=0$, 
which is a contradiction. 
\end{proof}

\medskip

\begin{thm} Let $X$ be an integral scheme over $(\mathbf C,\otimes,1)$. Then, 
$X$ is reduced and irreducible. 
\end{thm}

\begin{proof} Our argument is similar to the proof  
of \cite[Proposition 2.11]{1AB1}. It is clear that an integral scheme $X$ is also reduced. 
Suppose that $X$ is not irreducible, i.e., there exist
non-trivial Zariski open immersions $U=Spec(A)\longrightarrow X$ and 
$V=Spec(B)\longrightarrow X$ such that $U\times_XV$ is trivial. 
We consider the induced morphism $p:U\coprod V\longrightarrow X$.

\medskip 
Then, if $W=Spec(C)\longrightarrow X$ is any other Zariski open
immersion, it follows that
\begin{equation}
(U\times _XW)\times_W(V\times_XW)=(U\times_XV)\times_XW
\end{equation} is trivial. Since $C$ is an integral monoid, it follows from
Proposition \ref{p441} that $W=Spec(C)$ is also irreducible. Hence, at
least one of $U\times_XW$ and $V\times_XW$ must be trivial. We now 
construct the following pullback square
\begin{equation}\label{pbck1}
\begin{CD}
(U\times_XW)\coprod (V\times_XW) @>p_W>> W \\
@VVV @VVV \\
U\coprod V @>p>> X \\
\end{CD}
\end{equation} Since at least one of $U\times_XW$ and $V\times_XW$ 
is trivial, it follows that $p_W:(U\times_XW)\coprod (V\times_XW)
\longrightarrow W$ is equal to at least one of the Zariski
open immersions $U\times_XW\longrightarrow W$ or $V\times_XW
\longrightarrow X$. Hence, for any Zariski open immersion $W=Spec(C)\longrightarrow X$,
the morphism $p_W$ obtained from the pullback square \eqref{pbck1} is always a Zariski open 
immersion. It follows that $p:U\coprod V=Spec(A\oplus B)\longrightarrow X$ 
is a Zariski open immersion. Since $X$ is integral,
$\mathcal E(A\oplus B)$ must be an integral domain. However, if we 
consider the canonical morphisms
\begin{equation}
e_A:A\oplus B\longrightarrow A \longrightarrow A\oplus B 
\qquad e_B:A\oplus B\longrightarrow B\longrightarrow A\oplus B 
\end{equation} in $\mathcal E(A\oplus B)$, it is clear that
$e_A\circ e_B=0$. Hence, at least one of $A$ and $B$ is zero. 
This proves that $X$ is irreducible. 
\end{proof}

\medskip
Let $X$ be an integral scheme over $(\mathbf C,\otimes,1)$. We will now construct the analogue of the usual function field of $X$. Consider the collection of pairs $(U,t_U)$ such that
$U=Spec(A)\longrightarrow X$ is a Zariski open immersion, $A\ne 0$ and $t_U\in \mathcal E(A)$. Given non-trivial Zariski open
immersions $U=Spec(A)
\longrightarrow X$ and $V=Spec(B)\longrightarrow X$, we will say that two pairs
$(U,t_U)$ and $(V,t_V)$ are equivalent, written $(U,t_U)\sim (V,t_V)$, if there exists a Zariski immersion $W=Spec(C)
\longrightarrow U\times_XV$ such that the restrictions  of $t_U\in \mathcal E(A)$
and $t_V\in \mathcal E(B)$  to  $\mathcal E(C)$ are equal. 
Since $X$ is irreducible, the collection of these equivalence
classes defines an ordinary unital commutative ring, which we denote by $k(X)$. 

\medskip
\begin{thm} \label{prp3.5} Let $X$ be an integral scheme over $(\mathbf C,\otimes,1)$. Then, $k(X)$ is a field.  Further, let $f:Y\longrightarrow X$ be a dominant morphism of integral schemes, i.e., for any open immersion $V=Spec(B)\longrightarrow X$
with $B\ne 0$, the fibre product $Y\times_XV$ is non-trivial. Then, $f$ induces a morphism $k(f):k(X)\longrightarrow k(Y)$ of fields. 
\end{thm}

\begin{proof} Let us consider a pair $(U,t_U)$ defining a class in $k(X)$ with $U=Spec(A)$, $t_U\in \mathcal E(A)$, $t_U\ne 0$. Since $\mathcal E(A)$ is an integral domain, $A_{t_U}\ne 0$ as in the proof of Proposition \ref{p340}. From Corollary \ref{cor1}, 
$\mathcal E(A_{t_U})=\mathcal E(A)_{t_U}$ and hence we can consider the pair $(Spec(A_{t_U}),t_U^{-1})$ defining a 
class in $k(X)$. Then, as in \eqref{390e},
\begin{equation}
Spec(A_{t_U})\times_XSpec(A)\cong Spec(A_{t_U})\times_{Spec(A)}Spec(A)\cong Spec(A_{t_U})
\end{equation} and hence the product of the classes in $k(X)$ defined by $(U,t_U)$ and $(Spec(A_{U}),t_U^{-1})$
is unity. Hence, $k(X)$ is a field. 

\medskip
Now suppose that $f:Y\longrightarrow X$ is a dominant morphism of integral schemes as 
described above. We choose any pair $(V,t_V)$, $V=Spec(B)$, $t_V\in 
\mathcal E(B)$   defining a class in
$k(X)$ and consider $U:=V\times_XY$. Since $f$ is dominant, $U$ is 
non-trivial. Hence, we can choose an affine scheme $U'=Spec(A)$, $A\ne 0$  admitting a Zariski open immersion $U'\longrightarrow U$ into $U$.
We consider the composition $U'=Spec(A)\longrightarrow U\longrightarrow V=Spec(B)$
and let $g:B\longrightarrow A$ denote the corresponding morphism of monoids. We now associate the class defined by $(V,t_V)$ in $k(X)$ to the class 
defined by $(U',\mathcal E(g)(t_V))$ in $k(Y)$. It is clear that this
defines a morphism $k(f):k(X)\longrightarrow k(Y)$. 
\end{proof}

\medskip
Henceforth, for any integral scheme $X$ over $(\mathbf C,\otimes,1)$, we will say that $k(X)$ is the function field of $X$. We will
now prove that the function field of such a scheme is completely determined by any open 
subscheme. 

\medskip
\begin{cor}\label{c36} Let $X$ be an integral scheme and let
$U\longrightarrow X$ be a Zariski open immersion with $U$ non-trivial.
Then, we have an isomorphism of function fields $k(X)\cong k(U)$. 
\end{cor}

\begin{proof} We consider any pair $(V,t_V)$, $V=Spec(B)$, $t_V\in 
\mathcal E(B)$  defining a class in $k(X)$. Since $X$ is irreducible, the Zariski immersion $U\longrightarrow X$ is dominant. Hence, as in
the proof of Proposition \ref{prp3.5}, the pair $(V,t_V)$ determines a class
in $k(U)$ via the induced morphism $k(X)\longrightarrow k(U)$.  Conversely, consider any pair $(W,t_W)$, $W=Spec(C)$,
$t_W\in \mathcal E(C)$  defining a class in
$k(U)$. Then, using the composition $W\longrightarrow U\longrightarrow X$ of Zariski
open immersions, it follows that $(W,t_W)$ defines a class in $k(X)$. It is easy to check that these associations are inverses of each other and we have an isomorphism
$k(X)\cong k(U)$. 
\end{proof}

\medskip
Let $A$ be an  integral monoid and let $S\subseteq \mathcal E(A)$ be 
the multiplicatively closed subset of all nonzero elements of 
$\mathcal E(A)$. Then, we will always denote the localisation 
$A_S$ by $F(A)$. From Corollary \ref{cor1}, it follows that $\mathcal E(F(A))=
\mathcal E(A)_S$ and hence $\mathcal E(F(A))$ is the field of fractions
of the integral domain $\mathcal E(A)$. From Corollary \ref{c36}, it suffices to describe
the function field for integral schemes that are affine. Therefore, let $A$ be an integral
monoid in $(\mathbf C,\otimes,1)$ such that $Spec(A)$ is an integral scheme. We can now describe the function field of $Spec(A)$ more explicitly. 

\medskip
\begin{thm}\label{p37p} Let $A$ be an integral monoid object in $(\mathbf C,\otimes,1)$ 
such that $Spec(A)$ is an integral scheme. Then, the function field $k(Spec(A))$ can be described
as the filtered colimit
\begin{equation}\label{311e}
k(Spec(A))\cong \underset{B}{colim}\textrm{ }\mathcal E(F(B))
\end{equation} where the colimit in \eqref{311e} ranges over
all Zariski open immersions $Spec(B)\longrightarrow Spec(A)$
with $B\ne 0$. 
\end{thm}

\begin{proof} We note that if $Spec(B)\longrightarrow Spec(A)$ is a Zariski open immersion and
$Spec(A)$ is an integral scheme, it follows from Definition \ref{revD4.3} that $B$ is integral. If $B\ne 0$, then we can define
$F(B)$ and there is an induced morphism $\mathcal E(F(A))\longrightarrow \mathcal E(F(B))$. Since $Spec(A)$ is irreducible, it follows that the colimit in \eqref{311e} is filtered.
We let $C$ denote the colimit $C:=\underset{B}{colim}\textrm{ }\mathcal E(F(B))$.

\medskip
We now consider a pair $(U,t_U)$ with $U=Spec(B)$, $t_U\in
\mathcal E(B)$ defining a class in $k(Spec(A))$. Then, $Spec(B)$ admits a Zariski open immersion into $Spec(A)$ and we associate 
$(U,t_U)$ to the element of $C$ defined by $t_U\in \mathcal E(B)\subseteq \mathcal E(F(B))$.

\medskip Conversely, suppose that we choose any element $t\in \mathcal E(F(B))$ with $Spec(B)$ admitting a Zariski open immersion into $Spec(A)$. Then, $\mathcal E(F(B))$ is 
the field of fractions of $\mathcal E(B)$ and hence we may express
$t$ as $t=t_1t_2^{-1}$, with $t_1$, $t_2\in \mathcal E(B)$. Since
$\mathcal E(B_{t_2})=\mathcal E(B)_{t_2}\ne 0$, we have
$B_{t_2}\ne 0$. From Proposition \ref{openimmer}, we know that
$Spec(B_{t_2})\longrightarrow Spec(B)$ is a Zariski open immersion. 
We note that $t=t_1t_2^{-1}\in \mathcal E(B)_{t_2}=\mathcal E(B_{t_2})$. Hence, we can associate the class in $C$ defined by $t\in \mathcal E(F(B))$ to the class in $k(Spec(A))$ defined by the pair
$(Spec(B_{t_2}),t_1t_2^{-1})$. It is clear that these associations
are inverse to each other and hence we have an isomorphism
$k(Spec(A))\cong \underset{B}{colim}\textrm{ }\mathcal E(F(B))$ 
in \eqref{311e}. 

\end{proof}

\medskip

\medskip
\section{Closed subschemes and quasi-coherent sheaves of algebras}

\medskip

\medskip 
Since the theory of schemes over $(\mathbf C,\otimes,1)$ is developed by abstracting the properties
of Zariski open immersions in usual algebraic geometry, open subschemes  and open immersions fit naturally
into this formalism.  However, it is more difficult to develop an analogous notion of closed subschemes. This will
be the purpose of this section. We start by showing that, if a scheme $X$ is
semi-separated, there is a one-one
correspondence between quasi-coherent sheaves of algebras
on the  scheme $X$ and the collection of affine morphisms
$Y\longrightarrow X$. In particular, we apply this to construct
closed subschemes and the local ring corresponding 
to an integral subscheme of a Noetherian, integral and semi-separated  scheme. 

\medskip
\begin{defn} A scheme $X$ over $(\mathbf C,\otimes,1)$ will be said to be semi-separated if, given Zariski open immersions $Y_1\longrightarrow X$, $Y_2\longrightarrow X$ with $Y_1$, $Y_2$ affine, the fibre product $Y_1\times_XY_2$ is also an affine scheme. 
\end{defn}

\medskip

\begin{defn}
A morphism $f:Y\longrightarrow X$ of schemes over $(\mathbf C,\otimes,1)$ will be said to be 
affine if given a Zariski open immersion $U\longrightarrow X$
with $U$ affine, the fibre product $Y\times_XU$ is also affine. 
\end{defn}

\medskip
\begin{defn}\label{def4.2} Let $X$ be a scheme over $(\mathbf C,\otimes,1)$ and let $ZarAff(X)$ denote the
category of Zariski open immersions $U
\longrightarrow X$ with $U$ affine. Suppose that we have a functor:
\begin{equation}
\mathcal O:ZarAff(X)^{op}\longrightarrow Comm(\mathbf C)
\end{equation} For the sake of convenience, we will denote 
by $\mathcal O(U)$ the monoid associated to an object
$U\longrightarrow X$ in $ZarAff(X)$ by the functor
$\mathcal O$. We will say that $\mathcal O$ defines a quasi-coherent
sheaf of algebras on $X$ if $\mathcal O$ satisfies the following conditions:

\medskip
(a) For any object $U=Spec(A_U)\longrightarrow X$ in $ZarAff(X)$,
$\mathcal O(U)$ is an $A_U$-algebra. 

\medskip
(b) Let $f:V=Spec(A_V)\longrightarrow U=Spec(A_U)$ be a morphism in 
$ZarAff(X)$. Then, $\mathcal O(V)=\mathcal 
O(U)\otimes_{A_U}A_V$. 

\medskip
In particular, the quasi-coherent sheaf of algebras on $X$ defined
by associating an object $U=Spec(A_U)\longrightarrow X$ in 
$ZarAff(X)$ to $A_U\in Comm(\mathbf C)$ will be referred to as the
structure sheaf $\mathcal O_X$ of the scheme $X$.
\end{defn}

\medskip
We remark here that for (not necessarily Noetherian) schemes over
$(\mathbf C,\otimes,1)$, we have studied structures on the derived
category of quasi-coherent sheaves in \cite{2AB2}.

\medskip
\begin{thm}\label{affinemorp} Let $X$ be a semi-separated scheme over $(\mathbf C,\otimes,1)$. Then, there is a one-one
correspondence between quasi-coherent sheaves of algebras on $X$
and the collection of affine morphisms $Y\longrightarrow X$. 
\end{thm}

\begin{proof} First, we consider an affine morphism $f:Y\longrightarrow X$. Then, for any object $U=Spec(A_U)\longrightarrow X$
in $ZarAff(X)$,  $Y\times_XU$ must be affine and we 
let $Y\times_XU=Spec(B_U)$. Then, we can define a functor
\begin{equation}
\mathcal O:ZarAff(X)^{op}\longrightarrow Comm(\mathbf C)\qquad 
\mathcal O(U):=B_U
\end{equation} The induced morphism $Spec(B_U)=Y\times_XU
\longrightarrow U=Spec(A_U)$ ensures that $B_U$ is an 
$A_U$-algebra. Further, suppose that $V=Spec(A_V)
\longrightarrow Spec(A_U)=U$ is a morphism in $ZarAff(X)$. Then, we
have
\begin{equation}
Spec(B_V):=Y\times_XV=(Y\times_XU)\times_UV=Spec(B_U\otimes_{A_U}A_V)
\end{equation} and hence $\mathcal O(V)=B_V=B_U\otimes_{A_U}A_V=\mathcal O(U)\otimes_{A_U}A_V$. Hence, $\mathcal O$ is a quasi-coherent
sheaf of algebras on $X$. 

\medskip
Conversely, suppose that we are given a quasi-coherent sheaf 
$\mathcal O$ of algebras on $X$. Let us choose an affine cover
$\{U_i=Spec(A_i)\longrightarrow X\}_{i\in I}$ of $X$. For each 
$i\in I$, we set $B_i=\mathcal O(U_i)$ and $V_i=Spec(B_i)$. Since $X$
is semi-separated, the fibre products $U_i\times_XU_j$, $i,j\in I$ are all affine 
and we set $Spec(A_{i,j})=U_i\times_XU_j$. Since $\mathcal O$ is quasi-coherent, we know that
\begin{equation}
\begin{array}{c}
A_{i,j}\otimes_{A_i}B_i=A_{i,j}\otimes_{A_i}\mathcal O(U_i) =\mathcal O(U_i\times_XU_j)=A_{i,j}\otimes_{A_j}\mathcal O(U_j)=A_{i,j}\otimes_{A_j}B_j\\
\end{array}
\end{equation} Further, for any $i$, $j$, $k\in I$, we set
\begin{equation}
U_{ij}:=U_i\times_XU_j \qquad U_{ijk}=U_i\times_XU_j\times_XU_k
\end{equation}
We now
define
\begin{equation}\label{55rt}
Y':=\coprod_{i\in I}V_i\qquad 
R_{i,j}:=Spec(\mathcal O(U_i\times_XU_j))=Spec(A_{i,j}\otimes_{A_i}B_i)=Spec(A_{i,j}\otimes_{A_j}
B_j) 
\end{equation} for all $(i,j)\in I^2$. It is clear that the morphism 
\begin{equation}\label{424e}
\begin{array}{r}
R_{i,j}=Spec(A_{i,j}\otimes_{A_i}B_i)=Spec(A_{i,j})\times_{Spec(A_i)}
Spec(B_i)\longrightarrow \qquad \\ Spec(A_{i})\times_{Spec(A_i)}Spec(B_i)
= Spec(B_i)=V_i\\
\end{array}
\end{equation} obtained by base change from $U_i\times_XU_j=Spec(A_{i,j})\longrightarrow U_i=Spec(A_i)$
is a Zariski open immersion. Moreover, for any $i\in I$,
the morphism
\begin{equation}
R_{i,i}=Spec(A_{i,i}\otimes_{A_i}B_i)=Spec(A_i\otimes_{A_i}B_i)
=Spec(B_i)=V_i\longrightarrow 
Spec(B_i\otimes B_i)=V_i\times V_i
\end{equation} is identical to the ``diagonal morphism'' $V_i\longrightarrow V_i\times V_i$. Also, for
any $i,j,k\in I$, we consider
\begin{equation}\label{426e}
\begin{array}{ll}
R_{i,j}\times_{V_j}R_{j,k}&=Spec(A_{i,j}\otimes_{A_j}B_j\otimes_{B_j}
A_{j,k}\otimes_{A_k}B_k)\\
&=Spec(A_{i,j}\otimes_{A_j}A_{j,k}\otimes_{A_k}B_k)\\
&=Spec(A_{i,j})\times_{Spec(A_j)}Spec(A_j)\times_{X}Spec(A_k)\times_{Spec(A_k)}
Spec(B_k)\\
&=Spec(A_{i,j})\times_XSpec(A_k)\times_{Spec(A_k)}Spec(B_k) \\
&=(Spec(A_i)\times_XSpec(A_j)\times_XSpec(A_k))\times_{Spec(A_k)}
Spec(B_k)\\
&=U_{ijk}\times_{U_k}V_k
\end{array}
\end{equation} Proceeding in a manner similar to \eqref{426e}, we can show that 
\begin{equation}
R_{i,j}\times_{V_j}R_{j,k}=(Spec(A_i)\times_XSpec(A_j)\times_XSpec(A_k))\times_{Spec(A_i)}
Spec(B_i)=U_{ijk}\times_{U_i}V_i
\end{equation} From \eqref{55rt} and  \eqref{424e}, we know that 
\begin{equation}\label{427e}
\begin{array}{ll}
R_{i,k}&=(Spec(A_i)\times_XSpec(A_k))\times_{Spec(A_i)}Spec(B_i)=U_{ik}\times_{U_i}V_i\\
&=(Spec(A_i)\times_XSpec(A_k))\times_{Spec(A_k)}Spec(B_k)=U_{ik}\times_{U_k}V_k\\
\end{array}
\end{equation} Then, there exists a morphism $r_{ijk}:R_{i,j}\times_{V_j}R_{j,k}\longrightarrow R_{i,k}$ 
that may be described in either of the two following ways:
\begin{equation}\label{unnec}
\begin{CD}
R_{i,j}\times_{V_j}R_{j,k}=U_{ijk}\times_{U_i}V_i @>r_{ijk}>> U_{ik}\times_{U_i}V_i=R_{i,k} \\
@V=VV @V=VV \\
R_{i,j}\times_{V_j}R_{j,k}=U_{ijk}\times_{U_k}V_k @>r_{ijk}>> U_{ik}\times_{U_k}V_k=R_{i,k} \\
\end{CD}
\end{equation} Further, there is a  natural morphism $(q_i,q_k):R_{i,k}\longrightarrow V_i\times V_k$  induced by the pair
of morphisms:
\begin{equation}
q_i:R_{i,k}=U_{ik}\times_{U_i}V_i\longrightarrow U_i\times_{U_i}V_i=V_i \qquad q_k:R_{i,k}=U_{ik}\times_{U_k}V_k\longrightarrow U_k\times_{U_k}V_k=V_k
\end{equation} Similarly, there is a natural morphism $(p_i,p_k):R_{i,j}\times_{V_j}R_{j,k}\longrightarrow V_i\times V_k$  induced
by the pair of morphisms:
\begin{equation}
\begin{array}{c}
p_i:R_{i,j}\times_{V_j}R_{j,k}=U_{ijk}\times_{U_i}V_i:\longrightarrow U_i\times_{U_i}V_i =V_i \\
p_k:R_{i,j}\times_{V_j}R_{j,k}=U_{ijk}\times_{U_k}V_k:\longrightarrow U_k\times_{U_k}V_k =V_k \\
\end{array}
\end{equation} From the top row of \eqref{unnec} it is clear that $q_i\circ r_{ijk}=p_i$ and from the bottom row of
\eqref{unnec}, it is clear that $q_k\circ r_{ijk}=p_k$. It follows that we have
\begin{equation}
(q_i,q_k)\circ r_{ijk}=(p_i,p_k):R_{i,j}\times_{V_j}R_{j,k}\longrightarrow R_{ik}\longrightarrow V_i\times V_k
\end{equation} Hence, the natural morphism 
$R_{i,j}\times_{V_j}R_{j,k}\longrightarrow V_i\times V_k$
factors through $R_{i,k}$. Finally, since 
$R_{i,j}=R_{j,i}$ $\forall$ $i,j\in I$, it follows that
\begin{equation}
R=\coprod_{(i,j)\in I^2}R_{i,j}\subseteq Y'\times Y'
\end{equation} defines an equivalence relation on $Y'$ satisfying
the conditions of \cite[Proposition 2.18]{Toen2}. Hence, from
\cite[Proposition 2.18]{Toen2}, it follows that
$Y:=Y'/R$ defines a scheme $Y$ equipped with a natural morphism
$Y\longrightarrow X$. 

\medskip We now need  to show that the morphism $Y\longrightarrow X$ is affine. If $X=Spec(A)$ is already affine, then $Y=Spec(\mathcal O(X))$ and the result is clear. In general, we notice that given any Zariski 
open immersion $W\longrightarrow X$ with $W=Spec(C)$ affine, each of the operations used in constructing the scheme $Y$ above commutes
with the pullback to $W$. Hence, $W\times_XY=Spec(\mathcal O(W))$ and
the induced morphism $Y\longrightarrow X$ is affine.  Finally,  it may be verified that the two associations defined above are inverses of
each other.

\end{proof} 

\medskip
Let $X$ be a Noetherian  scheme. We consider a 
``quasi-coherent sheaf of ideals'' $\mathscr I$ on $X$, i.e., to
each object $U=Spec(A)\longrightarrow X$ in $ZarAff(X)^{op}$, we 
associate a proper ideal $\mathscr I(U)\subseteq \mathcal E(A)$ such
that given any Zariski open 
immersion $V=Spec(B)\longrightarrow U=Spec(A)$, 
\begin{equation}\label{517pcv}
\mathscr I(V)=\mathscr I(U)\otimes_{\mathcal E(A)}\mathcal E(B)
\end{equation} We will now show that given a quasi-coherent sheaf
of ideals $\mathscr I$ on $X$, we can associate a quasi-coherent
sheaf of algebras $\mathcal O_X/\mathscr I$ on $X$ in the sense
of Definition \ref{def4.2}. 

\medskip

\begin{thm}\label{4.13pr} Let $A$ be a Noetherian monoid and  let $f:A\longrightarrow A'$ be a morphism of  monoid objects in $(\mathbf C,\otimes,1)$
inducing a Zariski open immersion of affine schemes. Let $\mathscr I
\subseteq \mathcal E(A)$ be an ideal and let $\mathscr I'
\subseteq \mathcal E(A')$ be the ideal in $\mathcal E(A')$ extended
from $\mathscr I$ using the induced morphism $\mathcal E(f):
\mathcal E(A)\longrightarrow \mathcal E(A')$, i.e., $\mathscr I'=\mathscr I\otimes_{\mathcal E(A)}
\mathcal E(A')$. Then,
$A/\mathscr I\otimes_AA'\cong A'/\mathscr I'$.
\end{thm}

\begin{proof} Since $f:A\longrightarrow A'$ induces a Zariski open immersion, it follows from Proposition
\ref{2.3P} that $A'$ is also Noetherian. From Proposition \ref{2.6prp}, it follows that $\mathcal
E(A)$ and $\mathcal E(A')$ are Noetherian rings. Hence, we may choose
a finite set $\{t_1,...,t_k\}$ of generators for the ideal $\mathscr I\subseteq \mathcal E(A)$. Since
$\mathscr I'$ is extended from $\mathscr I$, it follows that
$\mathscr I'\subseteq \mathcal E(A')$ is generated by
$\{\mathcal E(f)(t_1),...,\mathcal E(f)(t_k)\}$. Then, as in  \eqref{239}, we know that
\begin{equation}\label{316e}
\begin{array}{c}
A/\mathscr I\cong A/t_1A\otimes_AA/t_2A\otimes_A\dots\otimes_AA/t_kA \\
A'/\mathscr I'\cong A'/\mathcal E(f)(t_1)A'\otimes_{A'}A'/\mathcal E(f)(t_2)A'\otimes_{A'}\dots\otimes_{A'}A'/\mathcal E(f)(t_k)A' \\
\end{array}
\end{equation} For any $t_i\in \mathcal E(A)$, we know that
\begin{equation}
A/t_iA:=colim(A\overset{t_i}{\longleftarrow}A\longrightarrow 0)
\end{equation} Further, we have:
\begin{equation}\label{318e}
A/t_iA\otimes_AA'=colim\left(\begin{CD}A'@<\mathcal E(f)(t_i)<< A'@>>>0\end{CD}\right)\cong A/\mathcal E(f)(t_i)A
\end{equation}
Finally, for any $A$-modules $M$ and $N$, it is clear
that 
\begin{equation}\label{319e}
(M\otimes_AA')\otimes_{A'}(A'\otimes_AN)\cong M\otimes_AA'\otimes_AN\cong (M\otimes_AN)\otimes_AA'
\end{equation} Combining \eqref{316e}, \eqref{318e} and using \eqref{319e}, 
it follows that $A/\mathscr I\otimes_AA'\cong A'/\mathscr I'$. 
\end{proof}

\medskip
From Proposition \ref{4.13pr} it follows that given a quasi-coherent
sheaf of ideals $\mathscr I$ on a Noetherian and semi-separated scheme $X$,
the functor
\begin{equation}
\mathcal O_X/\mathscr I:ZarAff(X)^{op}\longrightarrow Comm(\mathbf C) 
\qquad \mathcal O_X/\mathscr I(U):=\mathcal O_X(U)/\mathscr I(U)
\end{equation} defines a quasi-coherent sheaf of algebras on $X$ in the sense of 
Definition \ref{def4.2}. 
We denote by $Y_{\mathscr I}\longrightarrow X$ the affine morphism 
corresponding to the quasi-coherent sheaf of algebras $\mathcal O_X/
\mathscr I$ as described in Proposition \ref{affinemorp}. We will
refer to $Y_{\mathscr I}$ as a closed subscheme of $X$. 

\medskip
\begin{thm} Let $X$ be a Noetherian semi-separated scheme and let $Y_{\mathscr I}$
be a closed subscheme of $X$ corresponding to a quasi-coherent sheaf
of ideals $\mathscr I$ on $X$. Then, $Y_{\mathscr I}$ is Noetherian. 
\end{thm}

\begin{proof} We consider a Zariski affine covering $\{U_i=Spec(A_i)
\longrightarrow X\}_{i\in I}$ of $X$. Since $X$ is Noetherian, 
each $A_i$ is Noetherian. Then, as mentioned in the proof of Proposition 
\ref{affinemorp}, the following is a pullback square:
\begin{equation}
\begin{CD}
Spec(A_i/\mathscr I(U_i)) @>>> U_i=Spec(A_i) \\
@VVV @VVV \\
Y_{\mathscr I} @>>> X \\
\end{CD}
\end{equation} From Proposition \ref{noethlast}, we know that 
each $A_i/\mathscr I(U_i)$ is Noetherian. Now, given any Zariski 
open immersion $W=Spec(B)\longrightarrow Y_{\mathscr I}$, we
consider an affine covering $\{W_{ij}=Spec(B_{ij})\longrightarrow 
W\times_{Y_{\mathscr I}}Spec(A_i/\mathscr I(U_i))\}_{j\in J_i,i\in I}$ of each $W\times_{Y_{\mathscr I}}Spec(A_i/\mathscr I(U_i))$,
$i\in I$. Then, each $W_{ij}=Spec(B_{ij})$ admits a Zariski 
open immersion
\begin{equation}
Spec(B_{ij})=W_{ij}\longrightarrow W\times_{Y_{\mathscr I}}Spec(A_i/\mathscr I(U_i)) \longrightarrow Spec(A_i/\mathscr I(U_i))
\end{equation} From Proposition \ref{2.3P}, it follows that
each monoid $B_{ij}$, $j\in J_i$, $i\in I$ is Noetherian. Since the
collection $\{W_{ij}=Spec(B_{ij})\longrightarrow Spec(B)=W\}_{j
\in J_i,i\in I}$ is a Zariski covering, it follows from Proposition
\ref{Noeth2} that $B$ is Noetherian. 

\end{proof} 

\medskip Let $X$ be a Noetherian, integral, semi-separated scheme and let
$Y_{\mathscr I}$ be an integral closed subscheme corresponding
to a quasi-coherent sheaf of ideals $\mathscr I$ on $X$. We will now associate to $Y_{\mathscr I}$ a local
ring $\mathcal O_{Y_{\mathscr I}}$ that is analogous to the local ring at the generic point 
of an integral closed subscheme in usual algebraic geometry. 

\medskip For this, we consider
the collection of all pairs $(U,t_U)$ with $U$ an object of $ZarAff(X)$ such that
$U\times_XY_{\mathscr I}$ is non trivial and $t_U\in
\mathcal E(\mathcal O_X(U))$. We consider two such pairs $(U,t_U)$
and $(V,t_V)$. Since $U\times_XY_{\mathscr I}
\longrightarrow Y_{\mathscr I}$ and $V\times_XY_{\mathscr I}
\longrightarrow Y_{\mathscr I}$ are Zariski open immersions with
$U\times_XY_{\mathscr I}$ and $V\times_XY_{\mathscr I}$ non trivial
and $Y_{\mathscr I}$ is irreducible, it follows that
\begin{equation}
(U\times_XV)\times_XY_{\mathscr I}=(U\times_XY_{\mathscr I})
\times_{Y_{\mathscr I}}(V\times_XY_{\mathscr I})
\end{equation} is non trivial. Suppose that there exists a   Zariski affine 
covering $\{W_i\longrightarrow U\times_XV\}_{i\in I}$ and some $i_0\in I$ such that $W_{i_0}\times_XY_{\mathscr I}$ is non
trivial and the elements
in $\mathcal E(\mathcal O_X(W_{i_0}))$ corresponding to
$t_U\in \mathcal E(\mathcal O_X(U))$, $t_V
\in \mathcal E(\mathcal O_X(V))$ are equal. Then, we will say that
\begin{equation}
(U,t_U)\sim (V,t_V)
\end{equation} We note that since
$W_i\times_XY_{\mathscr I}$ forms a Zariski covering 
of $(U\times_XV)\times_XY_{\mathscr I}$, there exists  $
\phi\ne I'\subseteq I$ such that $W_i\times_XY_{\mathscr I}$
is non trivial for all $i\in I'$.  Then, $\sim$ is an 
equivalence relation and the collection of equivalence classes forms
a ring, which we denote by $\mathcal O_{Y_{\mathscr I}}$. 

\medskip
\begin{thm} Let $X$ be Noetherian, integral, semi-separated scheme and let
$Y_{\mathscr I}$ be an integral closed subscheme corresponding
to a quasi-coherent sheaf of ideals $\mathscr I$ on $X$. Then, 
$\mathcal O_{Y_{\mathscr I}}$ is a local ring. 
\end{thm}

\begin{proof} We let $\mathfrak m\subseteq \mathcal O_{Y_{\mathscr I}}$ denote the ideal consisting of all classes in $\mathcal O_{Y_{
\mathscr I}}$ induced by pairs $(U,t_U)$ such that $t_U
\in \mathscr I(U)$. Then, $1\notin \mathfrak m$. We now consider a pair $(V,t_V)$
inducing a class in $\mathcal O_{Y_{\mathscr I}}\backslash \mathfrak 
m$. From Proposition \ref{openimmer}, we know 
that $V'=Spec(\mathcal O_X(V)_{t_V})
\longrightarrow V=Spec(\mathcal O_X(V))$ is a Zariski open immersion
and it is clear that $t_V$ is a unit in $\mathcal E(\mathcal O_X(V)_{t_V})$. We now consider the fibre diagrams
\begin{equation}
\begin{CD}
Spec((\mathcal O_X(V)/\mathscr I(V))_{t_V}) @>>> Spec
(\mathcal O_X(V)/\mathscr I(V))@>>> Y_{\mathscr I} \\
@VVV @VVV @VVV \\
V'=Spec(\mathcal O_X(V)_{t_V}) @>>> V=Spec(\mathcal O_X(V)) 
@>>> X \\
\end{CD}
\end{equation} Since $Y_{\mathscr I}$ is integral,
$\mathcal E(\mathcal O_X(V)/\mathscr I(V))=\mathcal E(
\mathcal O_X(V))/\mathscr I(V)$ is an integral domain. Since $(V,t_V)$ induces a class in
$\mathcal O_{Y_{\mathscr I}}\backslash \mathfrak 
m$, we know that $t_V\notin \mathscr I(V)$ and hence $t_V$  induces a non-zero class in 
the integral domain $ \mathcal E(
\mathcal O_X(V))/\mathscr I(V)$. Hence, $\mathcal E((\mathcal O_X(V)/\mathscr I(V))_{t_V})=(\mathcal E(\mathcal O_X(V))/
\mathscr I(V))_{t_V}\ne 0$. Therefore, $V'\times_XY_{\mathscr I}=Spec((\mathcal O_X(U)/\mathscr I(U))_{t_V})$ is non trivial.

\medskip Further, it is clear that $(V',t_V)\sim (V,t_V)$. Since
$t_V$ must be a unit in $\mathcal E((\mathcal O_X(V)/\mathscr I(V))_{t_V})=(\mathcal E(\mathcal O_X(V))/
\mathscr I(V))_{t_V}$, the pair $(V',t_V)\sim (V,t_V)$ induces a 
class in $\mathcal O_{Y_{\mathscr I}}$ that is a unit. Hence, any element of $\mathcal O_{Y_{\mathscr I}}\backslash \mathfrak m$ is a unit. It follows that $\mathcal O_{Y_{\mathscr I}}$ is a local ring with maximal ideal $\mathfrak m$. 

\end{proof}

\medskip

\section{Examples}

\medskip

\medskip
In this section, we will present examples of categories $(\mathbf C,\otimes,1)$ over which
we can study algebraic geometry using the theory above. When $\mathbf C=R-Mod$, the 
category of modules over a commutative ring $R$, it is clear that our theory 
corresponds to the  usual algebraic geometry of schemes over $Spec(R)$. We will now show 
how to construct  other examples of such categories. 

\medskip
Let $X$ be a topological space and let $\mathcal A$ be a presheaf of commutative rings on $X$. We will
say that a presheaf $\mathcal M$ of abelian groups on $X$ is a presheaf
of  $\mathcal A$-modules
if it satisfies the following two conditions:

\medskip
(1) For any open  $U\subseteq X$, $\mathcal M(U)$ is an $\mathcal A(U)$ module.  

\medskip
(2) For open subspaces $V\subseteq U\subseteq X$, the induced morphism
$\mathcal M(U)\longrightarrow \mathcal M(V)$ is a morphism of $\mathcal A(U)$-modules,
where the $\mathcal A(V)$-module $\mathcal M(V)$ is treated as an $\mathcal A(U)$ module by restriction 
of scalars. 

\medskip Then, it is clear that the category $Premod(\mathcal A)$ of presheaves of $\mathcal A$-modules is an abelian symmetric monoidal category. 
In order to show that our theory can be applied to the category $\mathbf C=
Premod(\mathcal A)$, we need to check that it satisfies the conditions (C1) and (C2)
in Section 2. 

\medskip
We start by checking condition (C2). It is clear that a commutative monoid object $\mathcal B$ in 
$Premod(\mathcal A)$ is a presheaf of commutative $\mathcal A$-algebras; in particular, $\mathcal B$
is also a presheaf of commutative rings on $X$. Then, the category  $\mathcal B-Mod$ of  $\mathcal B$-modules in the symmetric monoidal category $Premod(\mathcal A)$ is identical
to the category $Premod(\mathcal B)$ of presheaves of $\mathcal B$-modules on $X$. 
From \cite[Corollary 2.15]{Prest}, it follows that $Premod(\mathcal B)$ is a 
locally finitely presented  Grothendieck abelian category and hence any object
in $\mathcal B-Mod=Premod(\mathcal B)$ can be written as a directed colimit of
finitely presented objects. 

\medskip
It remains to prove (C1). It is clear that the presheaf $\mathcal A$ is the ``unit
object'' for the symmetric monoidal structure on $Premod(\mathcal A)$. Moreover, for
any object $\mathcal M\in \mathbf C=Premod(\mathcal A)$, we have
\begin{equation}\label{ec6.2}
Hom_{Premod(\mathcal A)}(\mathcal A,\mathcal M)\cong \mathcal M(X)
\end{equation} Now, let
$\{\mathcal N_i\}_{i\in I}$ be an inductive system of objects in $Premod(\mathcal A)$
and let $\mathcal N:=colim_{i\in I}\mathcal N_i$. Then, by definition,
\begin{equation}\label{ec6.1}
\mathcal N(X)=\underset{i\in I}{colim}\textrm{ }\mathcal N_i(X)
\end{equation} From \eqref{ec6.2} and \eqref{ec6.1}, it follows that 
\begin{equation}
\underset{i\in I}{colim}\textrm{ }Hom_{Premod(\mathcal A)}(\mathcal A,\mathcal N_i)\cong Hom_{Premod(\mathcal A)}
(\mathcal A,\underset{i\in I}{colim}\textrm{ }\mathcal N_i)
\end{equation} In particular, $I$ could be a filtered inductive system or a finite
system. Hence, $\mathbf C=Premod(\mathcal A)$ satisfies condition (C1) as well.

\medskip
Hence, given a topological space $X$ and a presheaf $\mathcal A$ of commutative rings on $X$, the 
above theory enables us to do algebraic geometry in the category $Premod(\mathcal A)$
of presheaves of $\mathcal A$-modules on $X$. We end by mentioning several natural
examples of such situations. 

\medskip
(1) Let $X$ be a topological space and let $R$ be a commutative ring. We can take
$\mathcal A$ to be the constant presheaf of rings $R$ on $X$. Then, the category $Premod(\mathcal A)$ is the category of presheaves of $R$-modules on $X$. 

\medskip
(2) Let $X$ be a scheme. We can choose $\mathcal A$ to be the structure sheaf 
$\mathcal O_X$ of $X$. Then, the category $Premod(\mathcal A)$ is the category
of presheaves of $\mathcal O_X$-modules on $X$. 

\medskip
(3) Let $X$ be a topological space. We can define a presheaf $\mathcal A_{\mathbb R}$
(resp. a presheaf $\mathcal A_{\mathbb C}$) of rings on $X$  by setting 
$\mathcal A_{\mathbb R}(U)$ (resp. $\mathcal A_{\mathbb C}(U)$) to be the ring
of continuous real valued (resp. complex valued) functions on $U$, for any open set $U\subseteq X$. 

\medskip
(4) Let $X$ be a smooth (resp. complex) manifold. We can consider the presheaf 
$\mathcal A_{\mathbb R}^\infty$ (resp. $\mathcal A_{\mathbb C}^\infty$) of rings
by setting $\mathcal A_{\mathbb R}^\infty(U)$ (resp. $\mathcal A_{\mathbb C}^\infty(U)$)
to be the ring of infinitely differentiable real valued (resp. holomorphic complex valued) functions
on $U$, for any open set $U\subseteq X$. 

\medskip
(5) Let $X$ be a scheme of finite type over $\mathbb C$. Then, using the GAGA principle,
any Zariski open $U$ in $X$ corresponds to an analytic space $U^{an}$. Further,
since $X$ is a scheme of finite type, this association is functorial. Hence, we can consider
the presheaf $\mathcal A^{an}$ of rings defined by setting $\mathcal A^{an}(U)$
to be the ring of continuous complex valued functions on $U^{an}$ for 
any Zariski open $U$ in $X$.

\end{document}